\documentclass[12pt,a4paper]{amsart}
\usepackage{amsmath}
\usepackage{amssymb}
\usepackage{amscd,mathabx}
\usepackage{comment}
\usepackage{url}
\usepackage[all]{xypic}

\usepackage{tikz}
\usetikzlibrary{patterns}
\usepackage{graphicx}

\usepackage{todonotes}
\usepackage[a4paper, top=3cm, bottom=3cm, left=2.5cm, right=2.5cm]{geometry}

\newcommand{\BC}{\mathbb{C}}
\newcommand{\BZ}{\mathbb{Z}}
\newcommand{\BH}{\mathbb{H}}
\newcommand{\HFL}{\mathrm{HFL}^-}

\newcommand{\HF}{\mathrm{HF}^-}

\newcommand{\F}{\mathbb{F}}
\newcommand{\J}{\mathcal{J}}

\newcommand{\sss}{\ifmmode{{\mathfrak s}}\else{${\mathfrak s}$\ }\fi}
\newcommand{\sst}{\ifmmode{{\mathfrak t}}\else{${\mathfrak t}$\ }\fi}
\def\linkminus{\!\smallsetminus\!}

\def\Z{\mathbb Z}
\def\R{\mathbb R}
\def\Q{\mathbb Q}
\def\C{\mathbb C}
\def\F{\mathbb F}

\def\H{\mathbb H}

\def\qq{\mathbf{q}}
\def\rr{\mathbf{r}}
\def\JJ{\mathbf{J}}
\def\JFF{$J$-}
\def\HFF{$H$-}

\def\ss{\mathbf{s}}

\def\kk{\mathbf{k}}
\def\DD{\mathbf{D}}
\def\HH{\mathbf{H}}
\def\ggg{\mathbf{g}}
\def\vv{\mathbf{v}}
\def\uu{\mathbf{u}}
\def\LL{\mathcal{L}} 

\def\ww{\mathbf{w}}
\def\0{\mathbf{0}}
\def\mm{\mathbf{m}}
\def\aaa{\boldsymbol{a}}

\def\ee{\mathbf{e}}
\def\newm{\eta}
\def\lkv{\boldsymbol{\ell}}
\def\ddl{\boldsymbol{\theta}}

\def\wt#1{\widetilde{#1}}

\def\hf{\mathit{HF}}
\def\spinc{Spin$^c$}
\DeclareMathOperator\lk{lk}

\DeclareMathOperator\codim{codim}
\DeclareMathOperator{\Ord}{Ord}
\definecolor{c805080}{RGB}{128,80,128}
\definecolor{c508050}{RGB}{80,128,80}
\definecolor{c0000ed}{RGB}{0,0,237}
\definecolor{cffffd7}{RGB}{255,255,215}
\definecolor{c0000ff}{RGB}{0,0,255}
\definecolor{cffffff}{RGB}{255,255,255}

\newcommand{\bp}{\begin{pmatrix}}
\newcommand{\ep}{\end{pmatrix}}

\numberwithin{equation}{section}

\theoremstyle{plain}
\newtheorem{theorem}[equation]{Theorem}
\newtheorem*{theorem*}{Theorem}
\newtheorem{lemma}[equation]{Lemma}
\newtheorem{proposition}[equation]{Proposition}

\newtheorem{corollary}[equation]{Corollary}

\theoremstyle{definition}
\newtheorem{example}[equation]{Example}

\newtheorem{definition}[equation]{Definition}

\theoremstyle{remark}
\newtheorem{remark}[equation]{Remark}

\numberwithin{equation}{section}

\newtheorem*{ack}{Acknowledgements}

\title[Immersed concordances of links]{Immersed concordances of links and Heegaard Floer homology}
\author{Maciej Borodzik}
\address{Institute of Mathematics, Polish Academy of Science, ul. \'Sniadeckich 8, Warsaw, Poland}
\address{Institute of Mathematics, University of Warsaw, ul. Banacha 2,
02-097 Warsaw, Poland}
\email{mcboro@mimuw.edu.pl}

\author{Eugene Gorsky}
\address{Department of Mathematics, UC Davis, One Shields Avenue Davis CA 95616 USA}
\address{ National Research University Higher School of Economics, Vavilova 7, Moscow, Russia}
\email{egorskiy@math.ucdavis.edu}
\thanks{The second author was partially supported by the grants DMS-1559338 and  RFBR-13-01-00755}

\makeatletter
\@namedef{subjclassname@2010}{\textup{2010} Mathematics Subject Classification}
\makeatother

\begin{document}
\begin{abstract}
An immersed concordance between two links is a concordance with possible self-intersections. Given an immersed concordance we construct
a smooth four-dimensional cobordism between surgeries on links. 
By applying $d$-invariant inequalities for this cobordism we obtain inequalities between the \HFF functions of links, which 
can be extracted from the link Floer homology package. As an application we show a Heegaard Floer theoretical criterion for bounding the splitting number of links. The criterion is especially effective for L-space links, and we present an infinite family of L-space links with vanishing linking numbers and arbitrary large splitting numbers. We also show a semicontinuity of the \HFF function under $\delta$-constant deformations of singularities with many branches.
\end{abstract}
\maketitle

\section{Introduction}
\subsection{Overview}
An immersed cobordism between two links $\LL_1$ and $\LL_2$ in $S^3$ is a smoothly immersed surface in $S^3\times[1,2]$, whose boundary is
$\LL_2\subset S^3\times\{2\}$ and $\LL_1\subset S^3\times\{1\}$. An immersed concordance is an immersed cobordism, whose all the components have genus 0.
The notion of an immersed cobordism gives a unified approach for studying
smooth four genus, clasp number, splitting number and unlinking number of links. Recently many papers using this technique appeared
\cite{BFP, BL2, Kaw, NO, OS}. Generalizing the construction of \cite{BL2} we can use an
immersed concordance as a starting point in constructing 
a four-dimensional cobordism between large surgeries on $\LL_1$ and $\LL_2$ with precisely described surgery coefficients. Under some extra assumptions
we can guarantee that the four-dimensional cobordism is negative definite. We apply the the $d$-invariant
inequality of Ozsv\'ath and Szab\'o, see \eqref{eq:dinvariantinequality}, to relate the $d$-invariants of the corresponding surgeries on $\LL_1$ and $\LL_2$. These inequalities are best expressed in terms of the \HFF functions.

The \HFF function is a function that is used to calculate the $d$-invariant of large surgeries on links (see Theorem~\ref{thm:surgeryonlinks},
which can be thought of as an informal definition of $H$).
For knots it was first defined by Rasmussen in his thesis \cite{Ras} (as an analogue of the Fr\o yshov invariant in Seiberg-Witten theory), who used it
to obtain nontrivial bounds for the slice genus of knots. For L-space knots, the \HFF function can be easily reconstructed from the
Alexander polynomial. For L-space links with several components (see Section \ref{sec:L space}), the \HFF function was introduced by the second 
author and N\'emethi \cite{GN} (denoted by small $h$ there), who showed that for algebraic links it coincides with the Hilbert function
defined by the valuations on the local ring of the corresponding singularity. 

Unfortunately, apart from different notations of $H$ in the literature 
there are at least three different ``natural'' conventions on the definition of $H$, all differing by some shift of
the argument. This can be seen in \cite{BL},
where three different functions $I$, $J$ and $R$ denote very similar objects. In the link case the situation will be similar. The function called $H$
will take as an argument the levels of the Alexander filtration in the chain complex $CFL^-$, that is, its arguments will be from some lattice. Shifting 
the argument of $H$ by half the 
linking numbers will yield a function $J$ from $\Z^n$ to $\Z$. The normalization of the \JFF function makes it suit very well for
studying link concordances.  Finally, we will have
a function $R$, defined for algebraic singularities, which resembles the most the semigroup counting function from \cite{BL} and agrees with the Hilbert function from \cite{GN}.

We define the \HFF function for general links and
find inequalities between the \HFF functions of two links related by an immersed concordance (under some assumptions on the concordance).
The following theorem is one of the main results of the paper. The statement is easier in terms of the \JFF function than in terms of the \HFF function.
\begin{theorem*}[Theorem~\ref{thm:single}]
Let $\LL_1$ and $\LL_2$ be two $n$-component links differing by a single positive
crossing change, that is, $\LL_2$ arises by changing a negative crossing of $\LL_1$
into a positive one. Let $J_1$ and $J_2$ be the corresponding \JFF functions and let $\mm\in\Z^n$, $\mm=(m_1,\ldots,m_n)$.
\begin{itemize}
\item[(a)] If the crossing change is between two strands of the same component $L_{1i}$, then
\[J_2(m_1,m_2,\ldots,m_i+1,\ldots,m_n)\le J_1(m_1,\ldots,m_n)\le J_2(m_1,\ldots,m_i,\ldots,m_n).\]
\item[(b)] If the crossing change is between the $i$-th and $j$-th strand of $\LL_1$, then
\[J_2(m_1,m_2,\ldots,m_n)\le J_1(m_1,\ldots,m_n)\le J_2(m_1,\ldots,m_i-1,\ldots,m_n)\]
and 
\[J_2(m_1,m_2,\ldots,m_n)\le J_1(m_1,\ldots,m_n)\le J_2(m_1,\ldots,m_j-1,\ldots,m_n)\]
\end{itemize}
\end{theorem*}

As an application we provide new criteria for splitting numbers of links. 

\begin{theorem*}[Theorem \ref{th: vanishing}]
\
\begin{itemize}
\item[(a)] Suppose that a two component link  $\LL=L_1\cup L_2$ can be unlinked using $a$ positive and $b$ negative crossing changes.
Let $g_i$ denote the slice genus of $L_i$. Define vectors
$$
\ggg:=(g_1,g_2),\ \wt{\ggg}:=\left(g_1+\frac 12\lk(L_1,L_2),g_2+\frac 12\lk(L_1,L_2)\right).
$$
Define the region $R(a)$  by inequalities:
$$
R(a):=\{(m_1,m_2):m_1+m_2\ge a, m_1\ge 0, m_2\ge 0\}.
$$
Then $J(\mm)=\wt{J}(\mm)=0$ for $\mm\in R(a)+\ggg$.
\item[(b)] If, in addition, $\LL$ is an L-space link, then 
$$\HFL(\LL,\vv)=0\ \text{for}\ \vv \in R(a)+\wt{\ggg}+(1,1).$$
In particular, all coefficients of the Alexander polynomial vanish in $R(a)+\wt{\ggg}+(\frac 12,\frac 12)$.
\end{itemize}
\end{theorem*}

In the examples we focus on a family of two-bridge
links which were shown in \cite{Liu} to be L-space links.

\begin{theorem*}[Theorem \ref{th:two bridge}]
The splitting number of the two-component two-bridge link $$\LL_n=b(4n^2+4n,-2n-1)$$ equals $2n$, although the linking number between the components of $\LL_n$ vanishes.
\end{theorem*}

We compare this theorem with the work of Batson and Seed \cite{BS} which provides a different bound for the splitting number in terms of Khovanov homology. It turns out their lower bound is quite weak in this case and is at most three for all $\LL_n$.

Another application is a topological proof of semicontinuity of the Hilbert function of singularities under $\delta$-constant deformations.
The result was proved in \cite{BL2} for unibranched singular points (there is also an algebraic proof of a more general version in \cite{GN} for one
component links).  
Our result is for multibranched singularities under the assumption that the
number of branches does not change.

\subsection{Structure of the paper}

The paper uses a lot of background facts about Heegaard Floer homology and L-space links, most of them were discussed in \cite{GN,Liu}
using slightly different set of notations. For the reader's convenience, we repeat these facts and introduce the functions $H$ and $J$ in full generality in Sections 2 and 3. In Section 4, we relate the Ozsv\'ath-Szab\'o $d$-invariants of large surgeries on  a link to the \HFF function. Section 5 is the technical core of the paper: for an immersed cobordism between two links $\LL_1, \LL_2$ we construct a cobordism between the 
surgeries $S^3_{\qq_1}(\LL_1),S^3_{\qq_2}(\LL_2)$ of the 3-sphere on these links, and prove that it is negative definite under certain assumptions. In the negative definite case, we apply the classical inequality for $d$-invariants of $S^3_{\qq_1}(\LL_1),S^3_{\qq_2}(\LL_2)$, and obtain in Section 6 an inequality for $H$ and \JFF functions for the links $\LL_1,\LL_2$ stated in Theorem~\ref{thm:mainestimate}. We use this result 
to prove Theorem \ref{thm:single}.

In Section 7 we apply these results to obtain more concrete inequalities for two--component links, and prove Theorems \ref{th: vanishing} and \ref{th:two bridge}.
Finally, in the last section we apply the inequalities  to algebraic links and compare them with the algebro-geometric results on deformations of singularities.

\subsection{Notations and conventions}
All links are assumed to be oriented. 
For a link $\LL$, we denote by $L_1,\ldots,$ its components. This allows us to make a distinction between $\LL_1,\LL_2$ and $L_1,L_2$. The former
denotes two distinct links, the latter stays for two components of the same link $\LL$. 

We will mark vectors in  $n$--dimensional lattices in bold, in particular, we will write $\0=(0,\ldots,0)$. 
Given $\uu,\vv\in \BZ^n$, we write $\uu\preceq \vv$ if $u_i\leq v_i$ for all $i$,
and $\uu\prec \vv$ if $\uu\preceq \vv$ and $\uu\neq \vv$.
We will write $\ww=\max(\uu,\vv)$ (resp. $\ww=\min(\uu,\vv)$) if $w_i=\max(u_i,v_i)$ (resp. $w_i=\min(u_i,v_i)$) for all $i$.
We denote the $i$-th coordinate vector by $\ee_i$.

For a subset $I=\{i_1,\ldots,i_r\}\subset \{1,\ldots,n\}$ and $\uu\in \BZ^n$, we denote by $\uu_I$ the vector $(u_{i_1},\ldots,u_{i_r})$.
For a link $\LL=\bigcup_{i=1}^{n}L_i$ we denote by
$\LL_I=L_{i_1}\cup\ldots \cup L_{i_r}$ the corresponding sublink.

We will always work with $\F=\BZ/2\BZ$ coefficients.

\begin{ack}
The authors would like to thank to David Cimasoni, Anthony Conway,
Stefan Friedl, Jennifer Hom, Yajing Liu, Charles Livingston, Wojciech Politarczyk and Mark Powell for fruitful discussions.
The project was started during a singularity theory conference in Edinburgh in July 2015. The authors would like to thank the ICMS for
hospitality.
\end{ack}

\section{Links and L-spaces}

\subsection{Links and their Alexander polynomials}
Let $\LL\subset S^3$ be a link. Denote by $L_1,\ldots,L_n$ its components.
Throughout the paper, the multivariable Alexander polynomial (see \cite{Hi12} for definition) will be symmetric:
$$\Delta(t_1^{-1},\ldots,t_n^{-1})=\Delta(t_1,\ldots,t_n).$$
The sign of a multivariable Alexander polynomial can be fixed using the interpretation of the Alexander polynomial
via the sign refined Reidemeister torsion; see \cite[Section 4.9]{Hi12} for discussion and \cite{Tu01} for an introduction to Reidemeister torsion.

\begin{example}\label{ex:whitalex}
The Alexander polynomial for the Whitehead link equals 
$$\Delta_{Wh}(t_1,t_2)=-(t_1^{1/2}-t_1^{-1/2})(t_2^{1/2}-t_2^{-1/2}).$$ 
For
the Borromean link the Alexander polynomial equals 
$$\Delta_{Bor}(t_1,t_2,t_3)=(t_1^{1/2}-t_1^{-1/2})(t_2^{1/2}-t_2^{-1/2})(t_3^{1/2}-t_3^{-1/2}).$$
\end{example}

In some examples we will consider {\em algebraic links}, defined as intersections of complex plane curve singularities
with a small 3-sphere. The Alexander polynomials of algebraic links were computed by  Eisenbud and Neumann \cite{EN}.
In Section~\ref{sec:algebraiclinks} 
below we also discuss more recent results of Campillo, Delgado and Gusein-Zade \cite{CDG}, relating the Alexander polynomial to the 
algebraic invariants of a singularity, such as the multi-dimensional semigroup.

\begin{example}
The link of the singularity $x^2=y^{2n}$ consists of 2 unknots with linking number $n$. The corresponding Alexander polynomial equals
$$\Delta_{2,2n}(t_1,t_2)=\frac{t_1^{n/2}t_2^{n/2}-t_1^{-n/2}t_2^{-n/2}}{t_1^{1/2}t_2^{1/2}-t_1^{-1/2}t_2^{-1/2}}.$$
\end{example}

For future reference we recall the Torres formula, proved first in \cite{torres}. It relates the Alexander polynomial of a link $\LL$
with the Alexander polynomial of its sublink.
\begin{theorem}[Torres Formula]\label{thm:torres}
Let $\LL=L_1\cup\ldots\cup L_n$ be an $n$ component link and let $\LL'=L_1\cup\ldots\cup L_{n-1}$. The Alexander polynomials of $\LL'$ and of $\LL$
are related by the following formula.
\[
\Delta_{\LL}(t_1,\ldots,t_{n-1},1)=\begin{cases}
\left(\prod_{i=1}^{n-1}t_i^{\frac{1}{2}\lk(L_i,L_n)}-\prod_{i=1}^{n-1}t_i^{-\frac{1}{2}\lk(L_i,L_n)}\right)\Delta_{\LL'}(t_1,\ldots,t_{n-1})&\textrm{if $n>2$,}\\
\frac{\left(t_1^{\frac{1}{2}\lk(L_1,L_2)}-t_1^{-\frac{1}{2}\lk(L_1,L_2)}\right)}{\left(t_1^{\frac{1}{2}}-t_1^{-\frac{1}{2}}\right)}\Delta_{\LL'}(t_1)&\textrm{if $n=2$,}
\end{cases}
\]
where $\lk(L_i,L_n)$ is the linking number between $L_i$ and $L_n$.
\end{theorem}

\subsection{L-spaces and L-space links}
\label{sec:L space}
We will use the minus version of the Heegaard Floer link homology, defined in \cite{OSlinks}. To fix the conventions, we assume that 
$\HF(S^3)=\F[U]$ is supported in degrees $0, -2, -4,\ldots$.
To every 3-manifold $M$ this theory associates a chain complex $CF^{-}(M)$ which naturally splits as
a direct sum over \spinc{} structures on $M$: $CF^{-}(M)=\bigoplus_{\ss} CF^{-}(M,\ss)$.
The homology $\HF(M)=\bigoplus_{\ss} \HF(M,\ss)$, as a graded $\F[U]$-module, is a topological invariant of $M$.

\begin{definition}\
\begin{itemize}
\item[(a)] A 3-manifold $M$ is called an L-space if $b_1(M)=0$ and its Heegaard Floer homology has minimal possible rank: 
$\HF(M,\ss)\simeq \F[U]$ for all $\ss$.
\item[(b)] A link $\LL$ is called an L-space link if $S^3_{\qq}(\LL)$, the integral surgery of $S^3$ on the components of $\LL$
with coefficients $\qq=(q_1,\ldots,q_n)$, is an L-space for $\qq\ggcurly \0$.
\end{itemize}
\end{definition}

For a link $\LL=L_1\cup\ldots\cup L_n$ and a vector $\mm\in \BZ^n$ we define the framing matrix $\Lambda(\mm)=(\Lambda_{ij}(\mm))$:
\begin{equation}\label{eq:framingmatrix}
\Lambda_{ij}(\mm)=\begin{cases}
\lk(L_i,L_j) & \text{if}\ i\neq j,\\
m_i& \text{if}\ i=j.\\
\end{cases}
\end{equation}
It is well known that if $\det \Lambda\neq 0$ then $|H_1(S^3_{\mm}(\LL))|=|\det \Lambda(\mm)|$.
We recall the following result of Liu.
\begin{theorem}[see \expandafter{\cite[Lemma 2.5]{Liu}}]\label{thm:liusurg}
Suppose $\LL=L_1\cup\ldots\cup L_n$ is a link. Let $\mm=(m_1,\ldots,m_n)$ be a framing such that 
\begin{itemize}
\item[(a)] The framing matrix $\Lambda(\mm)$ is positive definite.
\item[(b)] For every $I\subset \{1,\ldots,n\}$ the $\mm_I$ surgery on $\LL_I$ is an L-space.
\end{itemize}
Then for any integer vector $\mm'\succeq \mm$ the $\mm'$ surgery on $\LL$
is an L-space. In particular, $\LL$ is an L-space link.
\end{theorem}

We will generalize this result for rational surgeries.
\begin{proposition}\label{prop:rationallarge}
Suppose $\LL$ and $\mm$ are as in the statement of Theorem~\ref{thm:liusurg}. 
Then for any rational framing vector $\qq\succeq \mm$, the $\qq$ surgery on $\LL$ is an L-space.
\end{proposition}
\begin{proof}
For a surgery vector $\qq$ denote by $t(\qq)$ the number of non-integer entries in the vector $\qq$.

Let us make the following statement.
\begin{equation}\tag{$I_{k,l}$}
\textit{For any $I\subset\{1,\ldots,n\}$ with $|I|\le l$, if $t(\qq_I)\le k$, then $S^3_{\qq_I}(\LL_I)$ is an L-space.}
\end{equation}
The statement $(I_{0,l})$ is covered for all $l\ge 1$ by Theorem~\ref{thm:liusurg}. 
Moreover, the statement $(I_{1,1})$ is standard.
Our aim is to show that $(I_{k,l})$ implies $(I_{k+1,l})$.

Choose $I\subset\{1,\ldots,n\}$ with $|I|=l$. 
Take $\qq\succeq\mm$ with $t(\qq)=k+1$. Suppose $j\in I$ is such that $q_j\notin\Z$ and let $I'=I\setminus\{j\}$.
Let $Y=S^3_{\qq_{I'}}(\LL_{I'})$.
As $t(\qq_{I'})=k$, the assumption $(I_{k,l-1})$ (which is contained in $(I_{k,l})$) implies that $Y$ is an L-space.
The component $L_j$ can be regarded as a knot in $Y$. Let $\mathcal{A}\subset\Q\cup\{\infty\}$ be the set of surgery coefficients
such that $a\in\mathcal{A}$ if and only if $Y_a(L_j)$ is an L-space. By the inductive assumption all integers $l\ge m_j$ belong to
$\mathcal{A}$, indeed $Y_l(L_1)$ is the surgery on $\LL$ with coefficient $\qq'_{I}$, where $\qq'_{I}$ is the
vector $\qq_I$ with $l$ at the $j$-th position.
Furthermore $\infty\in\mathcal{A}$
as well, because $Y$ itself is an L-space.

In \cite{RR} possible shapes of $\mathcal{A}$ were classified. The result allows us to conclude that if $m_1,m_1+1,\infty$ belong to $\mathcal{A}$,
then all rational numbers greater than $m_1$ are in $\mathcal{A}$. This shows $(I_{k+1,l})$.
\end{proof}

As an application of Proposition~\ref{prop:rationallarge} we will prove the following result, which
generalizes
\cite[Theorem 1.10]{hedcab}. 

\begin{proposition}\label{prop:cablesonlinks}
Suppose $\LL=L_1\cup\ldots\cup L_n$ is an L-space link. Let $p,q$ be coprime positive integers and let $\LL_{p,q}$ be the link 
$L_{cab}\cup L_2\cup\ldots\cup L_n$, where $L_{cab}$ is the $(p,q)$ cable on $L_1$. If $q/p$ is sufficiently large, than $\LL_{p,q}$
is also an L-space link. More precisely, if $\mm$ is an integer vector satisfying the conditions of Theorem \ref{thm:liusurg} then $\LL_{p,q}$ is an L-space link if $q/p>m_1$.
\end{proposition}
\begin{proof}
The proof is a direct generalization of \cite[Proof of Theorem 1.10]{hedcab}.
Choose $p$ and $q$ coprime and suppose that $\mm$ satisfies the conditions of Theorem \ref{thm:liusurg} and $q/p>m_1$.
First we will show that the $\mm'$ surgery on $\LL_{p,q}$
is an L-space, 
where $\mm'=(pq,m_2,\ldots,m_n)$. 
By \cite[Section 2.4]{hedcab} we know that $S^{3}_{\mm'}(\LL_{p,q})\simeq \textrm{Lens}(p,q)\# S^3_{\mm''}(\LL)$, 
where we set $\mm''=(q/p,m_2,\ldots,m_n)$ and
$\textrm{Lens}(p,q)$ is the lens space. As $\textrm{Lens}(p,q)$ is an L-space
and since a connected sum of L-spaces is an L-space, it is enough to show that $S^3_{\mm''}(\LL)$ is an L-space. 
But $\mm''\succeq \mm$, so by Proposition~\ref{prop:rationallarge} we conclude that $S^3_{\mm''}(\LL)$ is an L-space.
Hence the $\mm'$ surgery on $\LL_{p,q}$ is an L-space. The same proof applies to any sublink of $\LL_{p,q}$ which contains $L_{cab}$,
and for a sublink $\LL_I$ not containing $L_{cab}$ the $\mm'_I$--surgery is an L-space by assumption.

Let $\Lambda_{pq}$ be the framing
matrix for $\LL_{p,q}$ with framing $\mm'$, let $\Lambda$ be the framing matrix for $\LL$ with framing $\mm$. By assumption, $\Lambda$
is positive definite. The matrix $\Lambda_{pq}$ differs from $\Lambda$ only at the first column and at the first row. As
$\lk(L_{cab},L_j)=p\lk(L_1,L_j)$ for $j=2,\ldots,n$, we conclude that $\Lambda_{pq}$ can be obtained from $\Lambda$ by multiplying the first row and the first column by $p$ (the element in the top-left corner is multiplied by $p^2$) and then adding $qp-p^2m_1$ to the element in the top-left corner. The first operation is a matrix congruence so it preserves positive definiteness of the matrix. Adding an element can
be regarded as taking a sum with a matrix with all entries zero but $qp-p^2m_1$ in the top-left corner. This matrix is positive semi-definite,
because we assumed that $q/p>m_1$.
Now a sum of a positive definite matrix and a positive semi-definite one is a positive definite matrix. Therefore $\Lambda_{pq}$
is positive definite.

By Theorem~\ref{thm:liusurg} applied to $\LL_{p,q}$ with framing $\mm'$ we conclude that $\LL_{p,q}$ is an L-space link.
\end{proof}

To make Proposition \ref{prop:cablesonlinks} more concrete, we have to present an explicit vector $\mm$ satisfying the conditions of Theorem \ref{thm:liusurg}. This is done in the following theorem.

\begin{theorem}\label{thm:onDi}
Let $D_i$ denote the maximal degree of $t_i$ in the multivariable Alexander polynomial of an L-space link $\LL$,
$m_i\ge 2D_i+2$. Assume that $n>1$ and $\lk(L_i,L_j)\neq 0$ for all $i\neq j$. Then $\mm$ satisfies the conditions of Theorem \ref{thm:liusurg}.
\end{theorem}

\begin{proof}
Since the degrees of the multivariable Alexander polynomials of the sublinks of $\LL$ are less than $D_i$, it is sufficient to prove that $S^3_{\mm}(\LL)$ is an L-space and the framing matrix $\Lambda(\mm)$ is positive definite. The former is proved below as Lemma \ref{lem:D L space}.
To prove the latter, remark that by Theorem~\ref{thm:torres} one has:
\begin{equation}
\label{eq:projection}
\frac{\Delta(L_1)}{t^{1/2}-t^{-1/2}}=\frac{\Delta(t,1,\ldots,1)}{\prod_{j\neq 1}\left(t^{\frac{1}{2}\lk(L_1,L_j)}-t^{-\frac{1}{2}\lk(L_1,L_j)}\right)},
\end{equation}
so
$$
2D_i-\sum_{j\neq i}\lk(L_i,L_j)+1\ge 2\deg \Delta(L_i)\ge 0.
$$
Now $\Lambda(\mm)$ is a sum of $\binom{n}{2}$ positive definite matrices 
$$
\left(
\begin{matrix}
\lk(L_i,L_j) & \lk(L_i,L_j)\\
\lk(L_i,L_j) & \lk(L_i,L_j)\\
\end{matrix}
\right)
$$
with the only nonzero block at $i$-th and $j$-th rows and columns,
and a diagonal nonnegative definite matrix with entries 
$$m_i-\sum_{j\neq i}\lk(L_i,L_j)\ge 2D_i+2-\sum_{j\neq i}\lk(L_i,L_j)>0,$$
so it is positive definite.
\end{proof}

\begin{remark}
This bound is far from being optimal for links with many components. For example, it is proved in \cite{GHom} that
the point $(pq+1,\ldots,pq+1)$ satisfies the conditions of Theorem \ref{thm:liusurg} for the $(pn,qn)$ torus link, while in the above bound
one has $D_i=(npq-p-q)/2$ for $n>2$. On the other hand, for the $(2,2q)$ torus link we get $D_1=D_2=(q-1)/2$, so Theorem~\ref{thm:onDi} 
gives $m_i\ge q+1$, and the two bounds agree.  
\end{remark}


\section{Heegaard Floer link homology and the \HFF function for links}

In this section define the \HFF function for links and collect some useful facts about it.

\subsection{Alexander filtration}
A knot $K$ in a 3-manifold $M$ induces a filtration on the Heegaard Floer  complex $CF^{-}(M)$. 
Similarly, a link $\LL=L_1\cup\ldots \cup L_n$ 
with $n$ components in $M$ induces $n$ different filtrations on $CF^{-}(M)$, which can be interpreted as 
a filtration indexed by an $n$-dimensional lattice.
For a link in $S^3$, it is natural to make this lattice different from $\BZ^n$.

\begin{definition}\label{def:lattice}
Given an $n$-component oriented link $\LL\subset S^3$, we define an affine lattice over $\BZ^n$:
\[
\BH(\LL) =\bigoplus_{i=1}^{n}\BH_i(\LL),\qquad  \BH_i(\LL)=\BZ+\frac{1}{2}\lk(L_i,\LL\linkminus L_i).
\]
We also define the \emph{linking vector}:
\[\lkv=\lkv(\LL)=\frac12(\lk(L_1,\LL\linkminus L_1),\lk(L_2,\LL\linkminus L_2),\ldots,\lk(L_n,\LL\linkminus L_n))\]
We have $\BH(\LL)=\BZ^n+\lkv$.
\end{definition}

For $\vv\in \BH(\LL)$ define a subcomplex $A^{-}(\LL,\vv)\subset CF^-(S^3)$ corresponding to the filtration level $\vv$.
The filtration is ascending, so $A^{-}(\LL,\uu)\subset A^{-}(\LL,\vv)$ for $\uu\preceq \vv$.
The Heegaard Floer link homology $\HFL(\LL,\vv)$ can be defined as the homology of the associated graded complex:
\begin{equation}
\label{def of HFL}
\HFL(\LL,\vv)=H_{*}\left(A^{-}(\LL;\vv)/\sum_{\uu\prec \vv}A^{-}(\LL;\uu)\right).
\end{equation}
The Euler characteristic of this homology was computed in  \cite{OSlinks}:
\begin{multline}
\label{eq: delta tilde}
\wt{\Delta}(t_1,\ldots,t_n):=\sum_{\vv\in \BH(\LL)}\chi(\HFL(\LL,\vv))t_1^{v_1}\cdots t_n^{v_n}=\\
=\begin{cases}
(t_1\cdots t_n)^{1/2}\Delta(t_1,\ldots,t_n) & \text{if}\ n>1,\\
\Delta(t)/(1-t^{-1}) & \text{if}\ n=1,
\end{cases}
\end{multline}
where, as above, $\Delta(t_1,\ldots,t_n)$ denotes the  symmetrized Alexander polynomial of $\LL$.

One can forget a component $L_n$ in $\LL$ and consider the $(n-1)$-component link $\LL\linkminus L_n$. 
There is a natural forgetful map $\pi_n:\BH(\LL)\to \BH(\LL\linkminus L_n)$ defined by the equation:
\begin{equation}
\label{eq:forget}
\pi_n(v_1,\ldots,v_n)=\left(v_1-\frac{1}{2}\lk(L_1,L_n),\ldots,v_{n-1}-\frac{1}{2}\lk(L_{n-1},L_n)\right).
\end{equation}
In general, one defines a map $\pi_{\LL'}:\BH(\LL)\to \BH(\LL')$ for every sublink $\LL'\subset \LL$:
\begin{equation}
\label{eq:forget2}
[\pi_{\LL'}(v)]_{j}=\left(v_j-\lkv(\LL)_j+\lkv(\LL')_j\right)\quad \text{for}\ L_j\subset \LL'.
\end{equation}
Furthermore, for $v_n\gg 0$ the subcomplexes $A^{-}(\LL;\vv)$ stabilize, and by \cite[Proposition 7.1]{OSlinks} one has
a natural homotopy equivalence $A^{-}(\LL;\vv)\sim  A^{-}(\LL\linkminus L_n;\pi_n(\vv))$. 
More generally, for a sublink $\LL'=L_{i_1}\cup\ldots \cup L_{i_{n'}}$ one gets:
\begin{equation}
\label{projection for a-}
A^{-}(\LL';\pi_{\LL'}(\vv))\sim  A^{-}(\LL;\vv),\ \text{if}\ v_i\gg 0\ \text{for all}\ i\notin\{i_1\ldots,i_{r'}\}.
\end{equation}

There is an action of commuting operators $U_1,\ldots,U_n$ on the complex $A^{-}(\LL)$.
The action of $U_i$ drops the homological grading by $2$ and drops the $i$-th filtration level by $1$. In particular, 
$U_iA^{-}(\LL,\vv)\subset A^{-}(\LL,\vv-\ee_i)$.
This action makes the complexes $A^{-}(\LL,\vv)$ modules over the polynomial ring $\F[U_1,\ldots,U_n]$. It is known \cite{OSlinks} that $A^{-}(\LL,\vv)$ is a finitely generated module over $\F[U_1,\ldots,U_n]$, and all the $U_i$ are homotopic to each other on $A^{-}(\LL,\vv)$.
In particular, all the $U_i$ act in the same way in the homology $H_{*}(A^{-}(\LL,\vv))$, which can therefore be naturally considered as $\F[U]$--module,
where a single variable $U$ acts as $U_1$.

\subsection{The \HFF function}
It is known (see \cite{MO}, this is also a consequence of the Large Surgery Theorem~\ref{large surgery}  below) that 
the homology of $A^{-}(\LL,\vv)$ is isomorphic as an $\F[U]$-module to the Heegaard Floer homology of a large surgery on $\LL$ equipped with a certain \spinc{} structure.  Therefore it always splits as a direct sum of a single copy of $\F[U]$ and some $U$-torsion.
We begin with the following fact.

\begin{lemma}\label{lemnaturalinclusion}
For $\uu\preceq \vv$ the natural inclusion 
$$
\iota_{\uu,\vv}:A^{-}(\LL,\uu)\hookrightarrow A^{-}(\LL,\vv)
$$
is injective on the free parts of the homology, hence it is a multiplication by a nonnegative power of $U$. 
\end{lemma}
\begin{proof}
It is sufficient to prove that 
$$
\iota^{*}_{\vv-\ee_i,\vv}: H_{*}(A^{-}(\LL,\vv-\ee_i))\hookrightarrow H_{*}(A^{-}(\LL,\vv))
$$
is injective on the free parts. 
The latter holds
because $A^{-}(\LL,\vv-\ee_i)$ contains the image of $U_i\sim U$
acting on $A^{-}(\LL,\vv)$. Indeed, if $H_{*}(A^{-}(\LL,\vv))\simeq \F[U]+T(\vv)$, where $T(\vv)$ is $U$-torsion,
then $U\F[U]\subset H_{*}(U_iA^{-}(\LL,\vv))$. Consider the inclusions $U_iA^{-}(\LL,\vv)\subset A^{-}(\LL,\vv-\ee_i)\subset A^{-}(\LL,\vv)$. Since the composite inclusion of $U_iA^{-}(\LL,\vv)$ into $A^{-}(\LL,\vv)$ is injective on free parts,  we conclude that $\iota^{*}_{\vv-\ee_i,\vv}$ is injective and
\begin{equation}
\label{incl}
U\F[U]\subset \iota^{*}_{\vv-\ee_i,\vv}\F[U]\subset \F[U]
\end{equation}
\end{proof}
 
\begin{definition}
We define a function $H(\vv)=H_{\LL}(\vv)$ by saying that $-2H(\vv)$ is the maximal homological degree of the free part of $H_*(A^{-}(\LL,\vv))$.
\end{definition}

%
%
%
%

We will gather now some important properties of the \HFF function.
\begin{proposition}\label{prop:changeby1}
The function $H(\vv)$ has nonnegative integer values. Furthermore, for all $\vv\in \BH(\LL)$ one has $H(\vv-\ee_i)=H(\vv)$ or $H(\vv-\ee_i)=H(\vv)+1$. 
\end{proposition}

\begin{proof}
By Lemma~\ref{lemnaturalinclusion} the inclusion of $A^{-}(\LL,\vv)$ in $CF^{-}(S^3)$ induces an injective map on the free parts of 
the homology, so it sends a generator of the free part to $U^{k}$ times a generator of the free part for some $k\ge 0$.
 Since the inclusion preserves the homological grading (and the generator of $HF^-(S^3)$ has grading 0), 
the generator of the free part of $H_{*}(A^{-}(\LL,\vv))$ has grading $-2k$, and $k=H(\vv)$.
The last statement immediately follows from \eqref{incl}.
\end{proof}
\begin{proposition}\label{prop:splitlink}
If $\LL$ is a split link then $H(\vv)=\sum_{i=1}^{n} H_i(v_i)$, where $H_i$ is the \HFF function for the $i$-th component of the link.
\end{proposition}

\begin{proof}
For a split link by \cite[Section 11]{OSlinks} one has 
$$
A^{-}(\LL,\vv)\simeq A^{-}(L_1,v_1)\otimes_{\F[U]}\cdots \otimes_{\F[U]}A^{-}(L_n,v_n),
$$
 and the isomorphism preserves the homological gradings. 
Note that all the linking numbers of a split link vanish, so $\BH(\LL)=\BZ^n$, 
and the projections to sublattices do not require any shifts as in \eqref{projection for a-}.
\end{proof}

\begin{proposition}
For a sublink $\LL'=L_{i_1}\cup\ldots \cup L_{i_{r'}}$, one has
\begin{equation}
\label{projection for h}
H_{\LL}(\vv)=H_{\LL'}(\pi_{\LL'}(\vv))\ \text{if}\ v_i\gg 0\ \text{for}\ i\notin\{i_1\ldots,i_{r}\}.
\end{equation}
\end{proposition}

\begin{proof}
Follows from \eqref{projection for a-}.
\end{proof}

\subsection{The \HFF function for L-space links}\label{sec:hlspace}

By Theorem \ref{large surgery} (see also \cite{Liu}), a link is an L-space link if 
and only if $H_{*}(A^{-}(\vv))\simeq \F[U]$ for all $\vv\in\BH(\LL)$.
It turns out that for L-space links
the \HFF function is determined by the Alexander polynomial. 

Throughout Section~\ref{sec:hlspace} we will assume that $\LL$ is an L-space link.
Since $H_*(A^{-}(\vv))\simeq \F[U]$ for all $\vv\in\BH(\LL)$, by \eqref{def of HFL} and by the inclusion-exclusion formula one can write:
\begin{equation}
\label{chi from h}
\chi(\HFL(\LL,\vv))=\sum_{B\subset \{1,\ldots,n\}}(-1)^{|B|-1}H_\LL(\vv-\ee_B),
\end{equation}
where $\ee_B$ denotes the characteristic vector of the subset $B\subset \{1,\ldots,n\}$; see \cite[formula (3.3)]{GHom}. 
For $n=1$ equation \eqref{chi from h} has the form $\chi_{\LL,\vv}=H(\vv-1)-H(\vv)$, so $H(\vv)$ 
can be easily reconstructed from the Alexander polynomial: $H_\LL(\vv)=\sum_{\uu\ge \vv+1}\chi_{\LL,\uu}$. 
For $n>1$, one can also show that equation \eqref{chi from h} 
together with the boundary conditions \eqref{projection for h} has a unique solution, which is given by the following theorem: 

\begin{theorem}[\cite{GN}]
\label{th:invertion}
The \HFF function of an L-space link is determined by the Alexander polynomials of its sublinks as following:
\begin{equation}
\label{invertion}
H_{\LL}(v_1,\ldots,v_n)=\sum_{\LL'\subseteq \LL}(-1)^{\# \LL'-1}\sum_{\substack{\uu'\in\BH(\LL')\\ \uu'\succeq \pi_{\LL'}(\vv+\mathbf{1})}}\chi(\HFL(\LL',\uu')),
\end{equation}
where $\mathbf{1}=(1,\ldots,1)$.
\end{theorem}

There is a formula for the \HFF function in terms of the multivariable Alexander polynomial.
Consider the generating function:
\begin{equation}\label{eq:hhh}
\HH(t_1,\ldots,t_n):=\sum_{v_1,\ldots,v_n\in\Z^n} t_1^{v_1}\ldots t_n^{v_n}H(v_1,\ldots,v_n).
\end{equation}
Note that $\HH$ is a Laurent series in $t_i^{-1/2}$: by \eqref{projection for h} $H(\vv)$ vanishes if $v_i\gg 0$ for some $i$, but
it does not vanish for $\vv\llcurly \0$. Then \cite[Theorem 3.4.3]{GN} implies
\begin{equation}\label{eq:hhgen}
\HH(t_1,\ldots,t_n)=\prod_{i=1}^n\frac{1}{1-t_i^{-1}}\sum_{\LL'\subset \LL}(-1)^{\#\LL'-1}\wt{\Delta}_{\LL'}(t_{j_1},\ldots,t_{j_{\# \LL'}})
\prod_{j\colon L_j\subset \LL'} t_j^{\lkv(\LL)_j-\lkv(\LL')_j-1},
\end{equation}
where the sublink $\LL'$ has $r$ components and $\wt{\Delta}$ is defined by \eqref{eq: delta tilde}.
This immediately follows from \eqref{invertion} and the identity
$$
\sum_{\vv}t_1^{v_1}\ldots t_n^{v_n}\chi(\HFL(\LL',\pi_{\LL'}(\vv+\mathbf{1})))=\wt{\Delta}_{\LL'}(t_{j_1},\ldots,t_{j_{\#\LL'}})
\prod_{j\colon L_j\subset \LL'} t_j^{\lkv(\LL)_j-\lkv(\LL')_j-1},
$$
which itself follows from \eqref{eq: delta tilde} and \eqref{eq:forget2}.

As above, let $D_i$ denote the maximal $t_i$-degree of the Alexander polynomial of $\LL$, $\DD=(D_1,\ldots,D_n)$.
\begin{lemma}
\label{lem:H stabilize}
Assume that $\lk(L_i,L_j)\neq 0$, then $H(\vv)=H(\min(\vv,\DD)).$
\end{lemma}

\begin{proof} 
By Theorem~\ref{thm:torres} the degree of the Alexander polynomial of a sublink $\LL'=\LL_{I}$
in variable $t_i$ is less than or equal to $D_i-\frac{1}{2}\sum_{j\notin I}\lk(L_i,L_j)=\pi_{\LL'}(\DD)_i$.
Therefore if $v_i>D_i$ then 
$$\pi_{\LL'}(\vv)_i>\pi_{\LL'}(\DD)_i\ge \deg_{t_i}\Delta(\LL'),$$ 
and if $\uu\succeq \pi_{\LL'}(\vv)+\mathbf{1}$ then
$\chi(\HFL(\LL',\uu'))=0$. Therefore the summands contributing to \eqref{invertion} nontrivially 
correspond to subsets $I$ such that $v_i\le D_i$ for $i\in I$. Applying  \eqref{invertion}  to $\min(\vv,\DD)$,
one gets exactly the same summands.
\end{proof}

\begin{corollary}
For $\vv\succeq \DD$ one has $H(\vv)=0$.
\end{corollary}

\begin{lemma}
\label{lem:D L space}
Let $\DD$ be as above, then for $\mm\succeq 2\DD+2$ the surgery $S^3_{\mm}(\LL)$ yields an L-space.
\end{lemma}

\begin{proof} 
Consider the parallelepiped $P$ in $\BZ^n$ with opposite corners at $\DD$ and $-\DD$. To compute the
Heegaard Floer homology of $S^3_{\mm}(\LL)$, we use the surgery complex of Manolescu-Ozsv\'ath \cite{MO}.
Every \spinc{} structure on $S^3_{\mm}(\LL)$ corresponds to an equivalence class of $\BZ^n$ modulo the lattice generated by the columns of
$\Lambda(\mm)$. For $\mm\succeq 2\DD$ this equivalence class has at most one point in $P$, and the whole surgery complex can be contracted to a single copy of $\F[U]$ supported at that point. For the precise description of the ``truncation" procedure, we refer to \cite[Section 8.3, Case I]{MO}, where the constant $b$ in \cite[Lemma 8.8]{MO} can be chosen equal to $\DD$ by Lemma \ref{lem:H stabilize}.
\end{proof}

The following symmetry property of $H$, which generalizes the symmetry in the case of knots \cite{NiWu,HLZ},
is proved in \cite[Lemma 5.5]{Liu}.
\begin{proposition} 
For an L-space link one has 
\begin{equation}
\label{duality}
H(-\vv)=H(\vv)+|\vv|
\end{equation}
\end{proposition}

The symmetry \eqref{duality} and the projection formula \eqref{projection for h} imply a useful ``dual projection formula''.
\begin{corollary}
Let $\LL$ be an L-space link, consider a set $I=\{i_1,\ldots,i_r\}\subset \{1,\ldots,n\}$ 
and the sublink $\LL_I$. 
Then, as long as $v_j\ll 0$ for all $j\notin I$, the following holds:
\begin{equation}
\label{dual projection for h}
H_{\LL}(\vv)=H_{\LL_I}(\pi_{\LL_I}(\vv)+2\lkv(\LL)_{I}-2\lkv(\LL_I))-\sum_{j\notin I}v_j+|\lkv(\LL)_{I}-\lkv(\LL_I)|.
\end{equation}
\end{corollary}

\begin{proof}
For $\vv\in\BH(\LL)$ set $\vv_I=(v_{i_1},\ldots,v_{i_r})$.
By \eqref{duality}, $H_\LL(\vv)=H_\LL(-\vv)-|\vv|$. Since $-v_j\gg 0\ \text{for}\ j\notin I,$
the projection formula implies 
$$
H_\LL(-\vv)=H_{\LL_I}(\pi_{\LL_I}(-\vv))=H_{\LL_I}(-\vv_{I}-\lkv(\LL)_{I}+\lkv(\LL_I)),
$$ and \eqref{duality} for $\LL_I$ implies
\begin{multline*}
H_{\LL_I}(-\vv_{I}-\lkv(\LL)_{I}+\lkv(\LL_I)) =H_{\LL_I}(\vv_{I}+\lkv(\LL)_{I}-\lkv(\LL_I))+|\vv_{I}+\lkv(\LL)_{I}-\lkv(\LL_I)|=\\
H_{\LL_I}(\pi_{\LL_I}(\vv)+2\lkv(\LL)_{I}-2\lkv(\LL_I))+|\vv_{I}|+|\lkv(\LL)_{I}-\lkv(\LL_I)|.
\end{multline*}
\end{proof}

\subsection{The \JFF function}

The \JFF function of a link $\LL$ is essentially the same object as the \HFF function, only it differs from $H$ by a shift in variables.
This shift makes $J$
a function on $\Z^n$ instead of $\BH(\LL)$. It is therefore more convenient to study changes of the \JFF function under some changes (like
crossing changes) of the link $\LL$: these changes might affect the lattice $\BH(\LL)$. Yet another variant is the $\wt{J}$-function,
which turns out to be useful for bounding the splitting number of L-space links; see Section~\ref{sec:splitting} for
details.

\begin{definition}\label{def:J}
The \JFF function of a link $\LL$ with $n$ components is a function $J\colon\Z^n\to\Z$ given by
\[J(\mm)=H(\mm+\lkv),\ \mm\in\Z^n.\]
\end{definition}
With this definition the projection formula~\eqref{projection for h} takes a particularly simple form.
\begin{lemma}\label{lem:projectionJ}
Let $\mm\in\Z^n$ and $I\subset\{1,\ldots,n\}$.
Consider a sublink $\LL_I$ of $\LL$ and suppose that $m_i\gg 0$ for $i\notin I$.
Then we have
\[J_\LL(\mm)=J_{\LL_I}(\mm_I).\]
\end{lemma}  
\begin{proof}
Indeed, by \eqref{eq:forget2} and \eqref{projection for h}:
$$
J_\LL(\mm)=H_\LL(\mm+\lkv)=H_{\LL_I}(\mm_I+\lkv(\LL)_I-\lkv(\LL)_I+\lkv(\LL_I))=J_{\LL_I}(\mm_I).
$$
\end{proof}
 
In particular, the \JFF function of a component $L_i$ can be reconstructed from the values of \JFF function for $\LL$ 
evaluated on vectors  whose all components but the $i$-th one are sufficiently large.
 
\begin{definition}\label{def:wtj}
For $\mm=(m_1,\ldots,m_n)\in\Z^n$ define
\[\wt{J}(\mm)=J(\mm)-\sum_{i=1}^{n}J_{L_i}(m_i).\]
\end{definition}
The main feature of the $\wt{J}$-function is the following corollary to Proposition~\ref{prop:splitlink}.
\begin{corollary}\label{cor:splitwt}
If $\LL$ is a split link, then $\wt{J}=0$.
\end{corollary}

For general L-space links the $\wt{J}$-function can be calculated from the Alexander polynomial.
We have the following result.
\begin{proposition}\label{prop:generatingforL_j}
Define the generating function
\[\wt{\JJ}(t_1,\ldots,t_n)=\sum_{m_1,\ldots,m_n} t_1^{m_1}\ldots t_n^{m_n}\wt{J}(\mm).\]
Then $\wt{\JJ}$ is a Laurent series in $t_i^{-1/2}$ and the following equation holds:
\begin{equation}\label{eq:wtjtwo}
\wt{\JJ}(t_1,\ldots,t_n)=\prod_{i=1}^n\frac{1}{1-t_i^{-1}}\sum_{\substack{\LL'\subset \LL\\ \# \LL'>1}}(-1)^{\# \LL'-1}\Delta_{\LL'}(t_{j_1},\ldots,t_{j_{\# \LL'}})
\prod_{j\colon L_j\subset \LL'} t_j^{-\lkv(\LL')_j-\frac{1}{2}}.
\end{equation}
\end{proposition}
\begin{proof}
This is a consequence of previous definitions. 
The formula for the generating function for $J(\mm)$ immediately follows from \eqref{eq:hhgen}. To get a generating function for $\wt{J}$ we need
to subtract 
 the sum of \JFF functions for components $L_1,\ldots,L_n$ of $\LL$. We apply \eqref{eq:hhgen} again to calculate this contribution,
and remark that for $r>1$ one has $\wt{\Delta}_{\LL'}=\prod t_j^{1/2}\cdot \Delta_{\LL}'$.
\end{proof}

Equation~\eqref{eq:wtjtwo} takes a particularly simple form for a two-component link.
\begin{corollary}\label{cor:two}
For a link with two components 
$$\wt{\JJ}(t_1,t_2)=-\frac{(t_1t_2)^{-\lk(L_1,L_2)/2}\Delta(t_1,t_2)}{\left(t_1^{1/2}-t_1^{-1/2}\right)\left(t_2^{1/2}-t_2^{-1/2}\right)}.$$
\end{corollary}

\subsection{Examples}

\begin{example}\label{ex:whitehead}
Consider the Whitehead link. By \cite[Example 3.1]{Liu} it is an L-space link. The linking number  vanishes. The symmetrized  Alexander polynomial equals
$\Delta(t_1,t_2)=-(t_1^{1/2}-t_1^{-1/2})(t_2^{1/2}-t_2^{-1/2})$ (see Example~\ref{ex:whitalex}), 
so the nontrivial values of $\chi(\HFL(\vv))$ are
$$
\chi(\HFL(0,0))=-1,\ \chi(\HFL(1,0))=\chi(\HFL(0,1))=1,\ \chi(\HFL(1,1))=-1.
$$
Furthermore, both components are unknots, so $\chi(\HFL(\vv))=1$ for $\vv=(v,0)$ or $\vv=(0,v)$ with $v\le 0$, 
and $\chi(\HFL(\vv))=0$ for $\vv=(v,0)$ or $\vv=(0,v)$ for $v>0$.
The \HFF function of the components equals
\[
H_k(v_k)=\sum_{j\ge v_k+1}\chi(\HFL(j))=\max(-v_k,0)
\]
for $k=1,2$. By \eqref{invertion} we get
\begin{multline*}
H(v_1,v_2)=H_1(v_1)+H_2(v_2)-\sum_{\uu\succeq \vv+\mathbf{1}} \chi(\HFL(\uu))=\\=\begin{cases}
H_1(v_1)+H_2(v_2)+1=1,& \text{if}\ \vv=(0,0)\\
H_1(v_1)+H_2(v_2),& \text{otherwise}.
\end{cases}
\end{multline*}
\end{example}

\begin{example}
Consider the Borromean link. By \cite[Example 3.1]{Liu} it is an L-space link. The linking number between components vanishes, so do all bivariate Alexander polynomials of sublinks.
The trivariate symmetrized Alexander polynomial equals 
\[
\Delta(t_1,t_2,t_3)=(t_1^{1/2}-t_1^{-1/2})(t_2^{1/2}-t_2^{-1/2})(t_3^{1/2}-t_3^{-1/2}),
\] 
so
\[
\chi(\HFL(\LL,\vv))=\begin{cases}
1& \text{if}\ \vv=(1,1,1),(1,0,0),(0,0,1),(0,1,0)\\
-1& \text{if}\ \vv=(0,1,1),(1,1,0),(1,0,1),(0,0,0)\\
0& \text{otherwise}\\
\end{cases}
\]
By \eqref{invertion}, we get
\begin{multline*}
H(v_1,v_2,v_3)=\sum_{i=1}^{3}H_i(v_i)+\sum_{\uu\succeq \vv+\mathbf{1}} \chi(\HFL(\uu))=\\=\begin{cases}
\sum\limits_{i=1}^{3}H_i(v_i)+1=1& \text{if}\ \vv=(0,0,0)\\
\sum\limits_{i=1}^{3}H_i(v_i)& \text{otherwise}.
\end{cases}
\end{multline*}
\end{example}

\section{The \HFF function and $d$-invariants}

\subsection{Ozsv\'ath--Szab\'o $d$-invariants}
Let $Y$ be a rational homology three-sphere equipped with a \spinc{} structure $\sss$. The $d$-invariant of $(Y,\sss)$ is
the maximal grading of an element $x\in HF^-(Y,\sss)$, which maps non-trivially into $HF^\infty(Y,\sss)$. It is usually
a rational number.
The usefulness of the $d$-invariant comes from two facts: firstly it behaves well under a negative definite \spinc{} cobordism, secondly it can
be calculated from the knot (or link) Floer chain complex; we will describe this in detail in Section \ref{sec:surgeryonlinks} 
below. As for the behavior
under cobordism, suppose that $(Y_1,\sss_1)$ and $(Y_2,\sss_2)$ are rational homology 3-spheres and $W$ is a smooth 4-manifold
with boundary $Y_2\sqcup -Y_1$ endowed with a \spinc{} structure $\sst$ which restricts to $\sss_2$ on $Y_2$ and to $\sss_1$ on $Y_1$, put
differently, $(W,\sst)$ is a \spinc{} cobordism between $(Y_1,\sss_1)$ and $(Y_2,\sss_2)$.
Recall that for a rational homology sphere $\hf^\infty(Y,\sss)\cong \F[U,U^{-1}]$, where $U$ is a formal variable.
The following result is proved \cite{OzSz03}.
\begin{theorem}
There exists a map $F_{(W,\sst)}\colon \hf^{\infty}(Y_1,\sss_1)\to\hf^{\infty}(Y_2,\sss_2)$, commuting with multiplication
by $U$ and descending to maps $\hf^{-}(Y_1,\sss_1)\to\hf^{-}(Y_2,\sss_2)$ and 
$\hf^{+}(Y_1,\sss_1)\to\hf^{+}(Y_2,\sss_2)$ (which will be denoted by $F_{(W,\sst)}$).
The degree (the grading shift) of $F_{(W,\sst)}$ is
\begin{equation}\label{eq:degreedefinition}
\deg F_{(W,\sst)}=\frac14\left(c_1^2(\sst)-3\sigma(W)-2\chi(W)\right),
\end{equation}
where $c_1^2(\sst)$ is understood as the integral of the first Chern class over $W$ (it is a rational number in general). If additionally
$W$ is negative definite, that is, the intersection form on $H_2(W;\Q)$ is negative definite, then $F_{(W,\sst)}$ is an isomorphism on
$\hf^\infty$
and
\begin{equation}\label{eq:dinvariantinequality}
d(Y_2,\sss_2)\ge \deg F_{(W,\sst)} + d(Y_1,\sss_1).
\end{equation}
\end{theorem}
\begin{remark}\label{rem:additive}
It follows from definition that the degree is additive under the composition of \spinc{} cobordisms.
\end{remark}
The degree formula \eqref{eq:degreedefinition} will play an important role in this article. We will need the following fact, which is
well known to the experts.

\begin{proposition}\label{prop:blowup}
The degree $\deg F_{(W,\sst)}$ is preserved under negative
blow-ups. Namely, suppose $\pi\colon W'\to W$ is a blow-down map and $E$ is the exceptional
divisor. Let $\sst'$ be the \spinc{} structure of $W'$ whose first Chern class
is $\pi^*c_1(\sst)\pm PD[E]$, where $PD$ stays for Poincar\'e dual. Then 
\[\deg F_{(W',\sst')}=\deg F_{(W,\sst)}.\] 
On the other hand, if $\pi\colon W'\to W$ is a blow-down of a $(+1)$-sphere, then
\[\deg F_{(W',\sst')}=\deg F_{(W,\sst)}-1.\]
\end{proposition}
\begin{proof}
Let $s=[E]^2$, so $s=+1$ for the positive blow-up and $s=-1$ for the negative one. We have that $\chi(W')=\chi(W)+1$, $c_1^2(\sst')=c_1^2(\sst)+s$
and $\sigma(W')=\sigma(W)+s$. The change of the degree is $\frac14(s-3s-2)=-\frac12(s+1)$.
\end{proof}

\subsection{Large surgery theorem}\label{sec:surgeryonlinks}
The subcomplexes $A^{-}(\LL,\vv)$ are naturally related to the surgeries of the 3-sphere on $\LL$. Choose a framing vector
$\qq=(q_1,\ldots,q_n)$ such that $q_1,\ldots,q_n$ are sufficiently large. Let $\Lambda$
be the linking matrix of $\LL$, that is $\Lambda_{ij}=\lk(L_i,L_j)$ if $i\neq j$ and $\Lambda_{ii}=q_i$. 

Form a four-manifold $X_\qq$ by adding $n$ two-handles to a ball $B^4$: a handle with framing $q_i$ is attached along the component $L_i$.
The boundary $\partial X_{\qq}$, denoted $Y_{\qq}$, is the surgery on $\LL$ with framing $\qq$. Let $F_i$ be the surface obtained by gluing
a core of the $i$-th handle to a Seifert surface for $L_i$. By construction the classes $[F_1],\ldots,[F_n]$ freely generate
$H_2(X_\qq)$. With this choice of generators, we identify $H_2(X_\qq)$ with $\Z^n$.

Suppose $\det\Lambda\neq 0$. In this case $Y_\qq$ is a rational homology sphere.
In \cite[Section 8.5]{MO} there is given an enumeration of \spinc{} structures on $Y_\qq$, which
we are now going to recall. 

Fix $\zeta=(\zeta_1,\ldots,\zeta_n)$, a small real vector whose
entries are linearly independent over $\Q$. Then let $P(\Lambda)$ be the hyper-parallelepiped with vertices
\[\zeta+\frac12(\pm\Lambda_1\pm\Lambda_2\pm,\ldots,\pm\Lambda_n),\]
where all combinations of the signs are used and $\Lambda_1,\ldots,\Lambda_n$ are column vectors of the matrix $\Lambda$. 
Denote
\[P_\H(\Lambda)=P(\Lambda)\cap\H(\LL),\]
where $\H(\LL)$ is the lattice for $\LL$ as described in Definition~\ref{def:lattice} above.

\begin{proposition}[see \expandafter{\cite[Equation (125)]{MO}}]\label{prop:mo125}
For any $\vv\in P_\H(\Lambda)$ there exists a unique \spinc{} structure $\sss_\vv$ on $Y_\qq$ which extends to a \spinc{} structure $\sst_\vv$
on $X_\qq$ with $c_1(\sst_\vv)=2\vv-\left(\Lambda_1+\ldots+\Lambda_m\right)$.
\end{proposition}

\begin{theorem}[Large Surgery Theorem, see \expandafter{\cite[Section 10.1]{MO}}]
\label{large surgery}
Assume that $\qq\ggcurly\mathbf{0}$. For any $\vv\in P_\H(\Lambda)$,
the homology of  $A^{-}(\LL;\vv)$ is isomorphic to the Heegaard Floer homology of $S^3_{\qq}(\LL)$ with \spinc{} structure $\sss_\vv$. 
More precisely, we have an isomorphism over $\F[U]$:
\begin{equation}\label{eq:isomorphismofAandCF}
A^{-}(\LL,\vv)\simeq CF^{-}(S^{3}_{\qq}(\LL),\sss_\vv),
\end{equation}
where $U$ acts as $U_1$ on the left hand side.
In particular, all the $U_i$ are homotopic to each other on $A^{-}(\LL,\vv)$ and homotopic to $U$.
\end{theorem}

It is important to note that the isomorphism \eqref{eq:isomorphismofAandCF} shifts the grading.
The grading shift
can be calculated explicitly from the linking matrix $\Lambda$ and the vector $\vv$. We present a more geometric way, which
will suit best our applications.

Remove a small ball from $X_\qq$ and call the resulting manifold $U_\qq$. This is a cobordism between $S^3$ and $Y_\qq$. 
Let $U'_\qq=-U_\qq$. The \spinc{} structure $\sst_\vv$ on $U'_\qq$ gives a \spinc{} cobordism
between $(Y_\qq,\sss_\vv)$ and $S^3$ (equipped with the unique \spinc{} structure), so it induces a map $F_{(U'_\qq,\sst_\vv)}$ 
between Heegaard Floer homologies of $(Y_\qq,\sss_\vv)$ and $S^3$.
\begin{proposition}[see \expandafter{\cite[Section 10]{MO}}]\label{degreeshift}
The isomorphism \eqref{eq:isomorphismofAandCF} shifts the degree by $-\deg F_{(U'_\qq,\sst_\vv)}$.
\end{proposition}
As a corollary we can give a formula for $d$-invariants of large surgery on links.
\begin{theorem}\label{thm:surgeryonlinks}
For $\vv\in P_\H(\Lambda)$, the $d$-invariant of a surgery on $\LL$ is given by
\[d(S^3_\qq(\LL),\sss_\vv)=-\deg F_{(U'_\qq,\sst_\vv)}-2H(\vv).\]
\end{theorem}

\section{PSICs and four--manifolds}\label{sec:psic}

A positively self-intersecting concordance (later shortened to: a PSIC)  
is a generalization of the notion of a positively self-intersecting
annulus used in \cite{BL2} as a way to translate the questions about the unknotting number of knots
into problems of cobordisms of three-manifolds, where $d$-invariants can be used. The notion of a PSIC
concordance will play the same role for links.

\subsection{PSICs  in various guises.}

Let $\LL_1$ and $\LL_2$ be two links. Denote by $L_{11},\ldots,L_{1n_1}$ and $L_{21},\ldots,L_{2n_2}$ the components of $\LL_1$
and $\LL_2$, respectively. A \emph{positively self--intersecting concordance} (in short: PSIC) from $\LL_1$ to $\LL_2$
is a surface $A$, smoothly immersed into $S^3\times[1,2]$ such that $\partial A=\LL_2\sqcup -\LL_1$
and $A$ is topologically a union of immersed punctured disks. We require that $A$ has ordinary positive double points as singularities. 
Here and afterwards, whenever we write $\LL_{1}$ or $\LL_{2}$ it should be understood that $\LL_{1}\subset S^3\times\{1\}$ and 
$\LL_{2}\subset S^3\times\{2\}$. The same applies for the components of $\LL_1$ and $\LL_2$.

If $A$ is a PSIC, we denote by $A_1,\ldots,A_n$ its \emph{components}, that is the closures of connected components of 
$A\setminus \textrm{Sing }A$, where $\textrm{Sing }A$ denotes the set of singular points of $A$. 
Each of the $A_i$ is an immersed surface and $A=A_1\cup\ldots\cup A_n$. We define $\newm_{ij}=\#(A_i\cap A_j)$ for $i,j=1,\ldots,n$ and $i\neq j$;
for $i=j$ we set $\newm_{ii}$ to be the number of double
points of $A_i$. The total number of double points of $A$ is
\[p=\sum_{i\le j} \newm_{ij}.\]
Furthermore, set
\begin{equation}\label{eq:defai}
\aaa=(a_1,\ldots,a_n),\textrm{ where }a_i=\sum_{j\neq i} \newm_{ij}.
\end{equation}

The following specifications of the definition of a PSIC will be used in the present article:
\begin{itemize}
\item an \emph{annular} PSIC, shortly APSIC, is a PSIC such that each of the $A_i$ is an annulus such that $\partial A_i=L_{2i}\sqcup -L_{1i}$.
For an APSIC $n=n_2=n_1$. An exemplary APSIC is depicted in Figure~\ref{fig:conc1}.
\item a \emph{sprouting} PSIC, shortly SPSIC, is a PSIC such that for every $i=1,\ldots,n$, the intersection $A_i\cap S^3\times\{1\}=L_{1i}$.
For a SPSIC we have $n_2\ge n_1=n$. Furthermore, for any $i=1,\ldots,n$ we define the subset
\begin{equation}\label{eq:Thetai}
\Theta_i=\{j=1,\ldots,n\colon L_{2j}\subset\partial A_i\}.
\end{equation}
\item an \emph{elementary sprouting} PSIC, shortly ESPSIC, 
is a SPSIC such that $A$ is smooth and there exists $k\in\{1,\dots,n\}$ such that for $i\neq k$ we have $\Theta_i=\{i\}$ and $\Theta_k=\{k,n_2\}$.
For an EPSIC we have $n_2-1=n_1=n$.
\end{itemize}

\begin{figure}
\resizebox{10cm}{!}{
\begin{tikzpicture}[y=0.80pt,x=0.80pt,yscale=-1, inner sep=0pt, outer sep=0pt]
  \path[draw=black,line join=round,miter limit=4.00,even odd rule,line
    width=1.600pt,rounded corners=0.0000cm] (68.5714,383.7908) rectangle
    (535.0000,584.5050);
  \path[draw=c805080,line join=miter,line cap=butt,even odd rule,line
    width=0.800pt] (68.5714,457.3622) .. controls (104.1492,460.2512) and
    (140.7934,419.6639) .. (172.8571,411.6479) .. controls (258.1247,407.7759) and
    (510.4434,410.8623) .. (534.2857,412.3622) node [midway,above]{$A_1$};
  \path[draw=c508050,line join=miter,line cap=butt,even odd rule,line
    width=0.800pt] (68.5714,509.5051) .. controls (108.7183,509.7293) and
    (154.7504,509.5051) .. (196.4286,509.5051) .. controls (305.8967,509.5051) and
    (400.1078,525.1654) .. (501.4286,474.5051) .. controls (515.1817,467.6285) and
    (533.5714,473.0765) .. (533.5714,473.0765) node [midway,above] {$A_1$};
  \path[draw=c508050,line join=miter,line cap=butt,even odd rule,line
    width=0.800pt] (68.5714,552.3622) .. controls (142.2074,560.5773) and
    (219.4613,545.7248) .. (297.8571,552.3622) .. controls (381.4329,549.0175) and
    (459.1812,548.8674) .. (532.8571,557.3622) node [midway,above] {$A_2$};
  \path[draw=c805080,line join=miter,line cap=butt,even odd rule,line
    width=0.800pt] (68.5714,412.3622) .. controls (127.7067,410.6511) and
    (163.0729,465.6256) .. (214.1999,464.1773) .. controls (294.1701,468.5926) and
    (347.5995,470.3466) .. (405.7143,473.7908) .. node [near start,above] {$A_2$} controls (444.2669,472.0004) and
    (464.9813,514.1882) .. (533.5714,513.7908) ;
   \draw(300,600) node {$S^3\times[0,1]$};
   \draw(55,530) node {$L_{12}$};
   \draw(55,430) node {$L_{11}$};
\end{tikzpicture}}
\caption{An APSIC between two two-component links. Two crossings are
marked: one is a double point on an annulus, another one is
a crossing between two annuli.}\label{fig:conc1}
\end{figure}
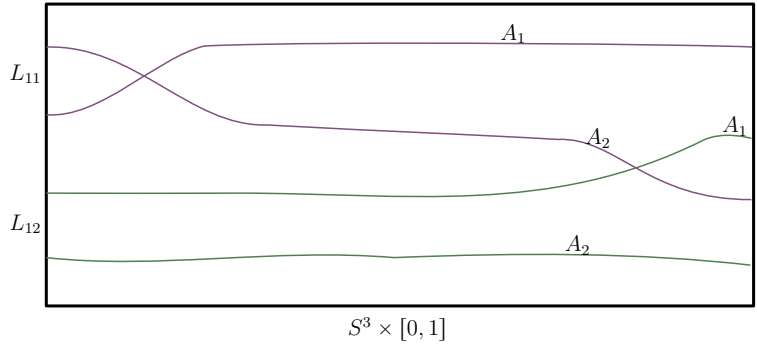
Let us introduce some useful terminology.
\begin{definition}
A double point $z$ of $A$ is called \emph{multicolored} respectively \emph{monochromatic}, if $z$ belongs to two components of $A$ (respectively, to
one component of $A$).
\end{definition}

For future reference we will need two simple facts. For simplicity, suppose $A$ is an APSIC (analogous statement can be proved
for general SPSIC, but we do not need it).
Define a shorthand 
\[l^{k}_{ij}=\lk(L_{ki},L_{kj}) \textrm{ for }i\neq j,\ k=1,2.\] 
As all the self-intersections of $A$
are positive and a positive self-intersection between different link components increases the linking number by $1$, 
we have that for $i\neq j$:
\begin{equation}\label{eq:linkingchange11}
l^2_{ij}=l^1_{ij}+\newm_{ij}.
\end{equation}
Equation~\eqref{eq:linkingchange11} summed up over $j\neq i$ yields the following result.
\begin{lemma}\label{lem:lkv}
Suppose that $A$ is an APSIC and let $\lkv_{1}$ and $\lkv_{2}$ be the linking vectors of $\LL_1$ and $\LL_2$ respectively. Then
\[\lkv_{2}-\lkv_{1}=\frac12\aaa.\]
\end{lemma}
\subsection{Topological constructions involving a PSIC.}\label{sec:topoconstruct}
In the following we generalize the construction of \cite{BL2} that based on a version of a PSIC for knots and as an output produced
a cobordism between surgeries of the two knots involved. We begin with a rather general construction,
later on we will specify its three
variants.  
The construction is done in four steps.

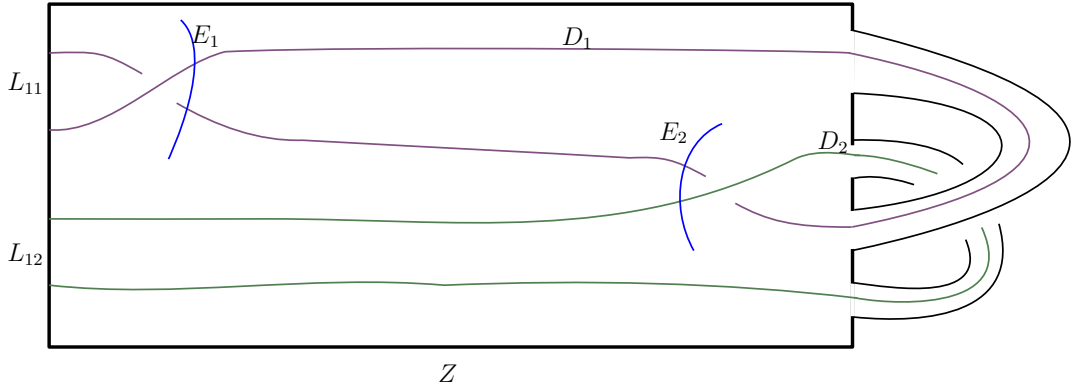
\begin{figure}
\resizebox{\hsize}{!}{
\begin{tikzpicture}[y=0.80pt,x=0.80pt,yscale=-1, inner sep=0pt, outer sep=0pt]
  \path[draw=black,line join=round,miter limit=4.00,even odd rule,line
    width=1.600pt,rounded corners=0.0000cm] (68.5714,383.7908) rectangle
    (535.0000,584.5050); 
  \path[draw=cffffff,line join=miter,line cap=butt,miter limit=4.00,even odd
    rule,line width=2.640pt] (534.4986,399.5410) .. controls (534.4986,395.2552)
    and (534.7908,435.7442) .. (534.7908,435.7442);
  \path[draw=cffffff,line join=miter,line cap=butt,miter limit=4.00,even odd
    rule,line width=4.000pt] (534.7937,466.4299) .. controls (534.7937,462.1442)
    and (534.9596,485.2080) .. (534.9596,485.2080);
  \path[draw=cffffff,line join=miter,line cap=butt,miter limit=4.00,even odd
    rule,line width=2.880pt] (534.7937,504.8788) .. controls (534.7937,504.8788)
    and (535.4335,512.1380) .. (535.2122,527.1293);
  \path[draw=cffffff,line join=miter,line cap=butt,miter limit=4.00,even odd
    rule,line width=1.920pt] (534.9831,547.3684) -- (535.0859,566.4622);

  \path[draw=c805080,line join=miter,line cap=butt,even odd rule,line
    width=0.800pt] (68.5714,457.3622) .. controls (104.1492,460.2512) and
    (138.7731,419.6639) .. (170.8368,411.6479) .. controls (256.1044,407.7759) and
    (509.6858,410.8623) .. (533.5281,412.3622) node [midway,above]{$D_1$};
  \path[draw=c508050,line join=miter,line cap=butt,even odd rule,line
    width=0.800pt] (68.5714,509.5051) .. controls (108.7183,509.7293) and
    (154.7504,509.5051) .. (196.4286,509.5051) .. controls (305.8967,509.5051) and
    (400.1078,525.1654) .. (501.4286,474.5051) .. controls (515.1817,467.6285) and
    (534.3290,472.3189) .. (539.3290,472.3189) node [midway,above]{$D_2$};
  \path[draw=c508050,line join=miter,line cap=butt,even odd rule,line
    width=0.800pt] (68.5714,548.3216) .. controls (142.2074,556.5367) and
    (219.4613,541.6842) .. (297.8571,548.3216) .. controls (381.4329,544.9769) and
    (459.9388,546.6577) .. (537.6147,555.8525);
  \path[draw=c805080,line join=miter,line cap=butt,even odd rule,line
    width=0.800pt] (68.5714,412.3622) .. controls (96.7105,411.5066) and
    (99.6216,411.2511) .. (122.6031,424.5997)(142.6031,441.7426) .. controls
    (165.5846,455.0913) and (190.0650,464.1871) .. (215.6285,463.4630) .. controls
    (295.5986,467.8783) and (347.5995,470.3466) ..
    (405.7143,473.7908)(405.7143,473.7908) .. controls (424.9906,472.8956) and
    (431.2359,474.4235) .. (449.8073,484.6970)(466.9502,500.4113) .. controls
    (485.5216,510.6848) and (500.2865,514.4946) .. (534.5816,514.2959);
  \path[draw=c0000ff,line join=miter,line cap=butt,even odd rule,line
    width=0.800pt] (145.0000,393.0765) .. controls (166.7351,413.6663) and
    (137.8571,474.5051) .. (137.8571,474.5051)node [very near start,right]{$E_1$};
  \path[draw=c0000ff,line join=miter,line cap=butt,even odd rule,line
    width=0.800pt] (459.2857,453.7908) .. controls (433.2760,464.8083) and
    (428.3551,502.0201) .. (442.8571,528.0765) node [very near start,left]{$E_2\phantom{A}$};
  \path[draw=c805080,line join=miter,line cap=butt,even odd rule,line
    width=0.800pt] (534.8267,514.2525) .. controls (682.4226,484.4669) and
    (661.7682,439.5452) .. (532.9921,412.3820);
  \path[draw=black,line join=miter,line cap=butt,even odd rule,line width=0.800pt]
    (536.0461,504.2166) .. controls (642.7809,492.6836) and (657.5606,441.6623) ..
    (535.5843,435.9336);
  \path[draw=black,line join=miter,line cap=butt,even odd rule,line width=0.800pt]
    (534.0498,398.8636) .. controls (740.4139,448.7613) and (661.8577,500.9062) ..
    (535.7304,527.8475);
  \path[draw=black,line join=miter,line cap=butt,even odd rule,line width=0.800pt]
    (534.8539,463.3884) .. controls (564.8669,462.5111) and (588.4935,467.4954) ..
    (599.1396,477.6742);
  \path[draw=black,line join=miter,line cap=butt,even odd rule,line width=0.800pt]
    (536.0027,485.7244) .. controls (551.7509,482.5186) and (570.7143,489.5051) ..
    (570.7143,489.5051);
  \path[draw=black,line join=miter,line cap=butt,even odd rule,line width=0.800pt]
    (600.7143,521.6479) .. controls (616.1830,559.3234) and (558.9229,550.0555) ..
    (535.2451,546.8336);
  \path[draw=c508050,line join=miter,line cap=butt,even odd rule,line
    width=0.800pt] (610.0000,514.5051) .. controls (634.1061,567.7248) and
    (548.2067,558.5900) .. (536.3184,555.6180);
  \path[draw=black,line join=miter,line cap=butt,even odd rule,line width=0.800pt]
    (620.0000,512.3622) .. controls (634.2541,565.1857) and (579.1471,571.3341) ..
    (535.7935,566.9834);
  \path[draw=c508050,line join=miter,line cap=butt,even odd rule,line
    width=0.800pt] (536.0460,472.3622) .. controls (555.1726,471.6872) and
    (584.2857,483.0765) .. (584.2857,483.0765);
   \draw(300,600) node {$Z$};
   \draw(55,530) node {$L_{12}$};
   \draw(55,430) node {$L_{11}$};
\end{tikzpicture}}
\caption{Step 2 of the construction in Section~\ref{sec:topoconstruct}. The boundary of $Z$ is $S^3$ on the left and $S^3_{\qq_2}(\LL_2)$
on the right.}\label{fig:conc2}
\end{figure}
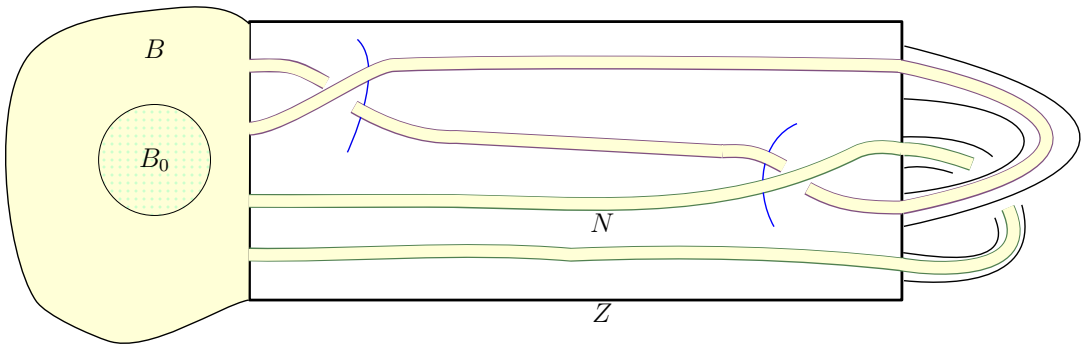
\begin{figure}
\resizebox{\hsize}{!}{
\begin{tikzpicture}[y=0.80pt,x=0.80pt,yscale=-1, inner sep=0pt, outer sep=0pt, every node/.style={scale=1.3}]
  \begin{scope}[shift={(-0.70711,-3.53553)}]
    \path[draw=black,line join=round,miter limit=4.00,even odd rule,line
      width=1.600pt,rounded corners=0.0000cm] (68.5714,383.7908) rectangle
      (535.0000,584.5050);
  \end{scope}
  \path[draw=c805080,line join=miter,line cap=butt,miter limit=4.00,even odd
    rule,line width=8.000pt] (67.9364,457.3622) .. controls (92.3074,457.4822) and
    (152.9530,412.1981) .. (170.8368,411.6479) .. controls (256.1044,407.7759) and
    (509.6858,410.8623) .. (533.5281,412.3622);
  \path[draw=c508050,line join=miter,line cap=butt,miter limit=4.00,even odd
    rule,line width=8.000pt] (67.5893,509.5051) .. controls (105.5933,509.7293)
    and (151.6254,509.5051) .. (193.3036,509.5051) .. controls (302.7717,509.5051)
    and (400.1078,525.1654) .. (501.4286,474.5051) .. controls (515.1817,467.6285)
    and (534.3290,472.3189) .. (534.3290,472.3189);
  \path[draw=c508050,line join=miter,line cap=butt,miter limit=4.00,even odd
    rule,line width=8.000pt] (67.6250,548.0716) .. controls (137.9574,549.6617)
    and (219.4613,541.6842) .. (297.8571,548.3216) .. controls (381.4329,544.9769)
    and (461.0638,546.9077) .. (534.7397,555.4025);
  \path[draw=c805080,line join=miter,line cap=butt,miter limit=4.00,even odd
    rule,line width=8.000pt] (67.7346,412.1097) .. controls (98.7308,411.2541) and
    (99.6216,411.2511) .. (122.6031,424.5997)(142.6031,441.7426) .. controls
    (165.5846,455.0913) and (190.0650,464.1871) .. (215.6285,463.4630) .. controls
    (295.5986,467.8783) and (347.5995,470.3466) ..
    (405.7143,473.7908)(405.7143,473.7908) .. controls (424.9906,472.8956) and
    (431.2359,474.4235) .. (449.8073,484.6970)(466.9502,500.4113) .. controls
    (485.5216,510.6848) and (500.2865,514.4946) .. (534.5816,514.2959);
  \path[draw=c0000ed,line join=miter,line cap=butt,even odd rule,line
    width=0.800pt] (145.0000,393.0765) .. controls (166.7351,413.6663) and
    (137.8571,474.5051) .. (137.8571,474.5051);
  \path[draw=c0000ff,line join=miter,line cap=butt,even odd rule,line
    width=0.800pt] (459.2857,453.7908) .. controls (433.2760,464.8083) and
    (428.3551,502.0201) .. (442.8571,528.0765);
  \path[draw=c805080,line join=miter,line cap=butt,miter limit=4.00,even odd
    rule,line width=8.000pt] (534.8267,514.2525) .. controls (682.4226,484.4669)
    and (661.7682,439.5452) .. (532.9921,412.3820);
  \path[draw=black,line join=miter,line cap=butt,even odd rule,line width=0.800pt]
    (536.0461,504.2166) .. controls (642.7809,492.6836) and (657.5606,441.6623) ..
    (535.5843,435.9336);
  \path[draw=black,line join=miter,line cap=butt,even odd rule,line width=0.800pt]
    (536.0498,397.8636) .. controls (740.4139,448.7613) and (661.8577,500.9062) ..
    (535.7304,527.8475);
  \path[draw=black,line join=miter,line cap=butt,even odd rule,line width=0.800pt]
    (534.8539,463.3884) .. controls (564.8669,462.5111) and (588.4935,467.4954) ..
    (599.1396,477.6742);
  \path[draw=black,line join=miter,line cap=butt,even odd rule,line width=0.800pt]
    (536.0027,485.7244) .. controls (551.7509,482.5186) and (570.7143,489.5051) ..
    (570.7143,489.5051);
  \path[draw=black,line join=miter,line cap=butt,even odd rule,line width=0.800pt]
    (600.7143,521.6479) .. controls (616.1830,559.3234) and (558.9229,550.0555) ..
    (535.2451,546.8336);
  \path[draw=c508050,line join=miter,line cap=butt,miter limit=4.00,even odd
    rule,line width=8.000pt] (610.0000,514.5051) .. controls (634.1061,567.7248)
    and (545.8317,558.0900) .. (533.9434,555.1180);
  \path[draw=black,line join=miter,line cap=butt,even odd rule,line width=0.800pt]
    (620.0000,512.3622) .. controls (634.2541,565.1857) and (579.1471,571.3341) ..
    (535.7935,566.9834);
  \path[draw=c508050,line join=miter,line cap=butt,miter limit=4.00,even odd
    rule,line width=8.000pt] (536.0460,472.3622) .. controls (555.1726,471.6872)
    and (584.2857,483.0765) .. (584.2857,483.0765);
    \path[draw=cffffd7,line join=miter,line cap=butt,miter limit=4.00,even odd
      rule,line width=6.720pt] (67.3326,457.4052) .. controls (95.9847,455.8022) and
      (151.8045,412.5768) .. (170.7382,411.4384) .. controls (256.0057,407.5663) and
      (509.5872,410.6527) .. (533.4294,412.1526);
  \path[draw=cffffd7,line join=miter,line cap=butt,miter limit=4.00,even odd
    rule,line width=6.720pt] (67.3812,412.2014) .. controls (97.8723,410.4166) and
    (100.8954,413.0323) .. (122.8668,424.4823)(142.8668,441.6251) .. controls
    (165.8483,454.9738) and (190.3287,464.0697) .. (215.8922,463.3455) .. controls
    (295.8623,467.7608) and (347.8632,470.2291) ..
    (405.9780,473.6733)(405.9780,473.6733) .. controls (425.2543,472.7781) and
    (431.4996,474.3060) .. (450.0711,484.5795)(467.2139,500.2938) .. controls
    (485.7853,510.5673) and (500.5502,514.3771) .. (534.8453,514.1784);
  \path[draw=cffffd7,line join=miter,line cap=butt,miter limit=4.00,even odd
    rule,line width=6.720pt] (533.6332,514.1216) .. controls (683.7545,484.0834)
    and (661.0797,439.4142) .. (532.3036,412.2511);
  \path[draw=cffffd7,line join=miter,line cap=butt,miter limit=4.00,even odd
    rule,line width=6.720pt] (67.5463,509.5519) .. controls (105.5503,509.7761)
    and (154.5824,509.3019) .. (196.2606,509.3019) .. controls (305.7287,509.3019)
    and (399.9398,524.9622) .. (501.2606,474.3019) .. controls (515.0137,467.4253)
    and (534.1610,472.1157) .. (534.1610,472.1157);
  \path[draw=cffffd7,line join=miter,line cap=butt,miter limit=4.00,even odd
    rule,line width=6.720pt] (533.2718,472.1921) .. controls (551.8626,470.8028)
    and (584.1901,482.9064) .. (584.1901,482.9064);
  \path[draw=cffffd7,line join=miter,line cap=butt,miter limit=4.00,even odd
    rule,line width=6.720pt] (67.3139,548.3545) .. controls (137.5213,549.4446)
    and (218.9002,541.5921) .. (297.2960,548.2295) .. controls (380.8718,544.8848)
    and (461.1277,546.6906) .. (534.8036,555.1854);
  \path[draw=cffffd7,line join=miter,line cap=butt,miter limit=4.00,even odd
    rule,line width=6.720pt] (610.1240,514.3347) .. controls (634.2300,567.5545)
    and (545.4556,558.0447) .. (533.5674,555.0726);
  \path[draw=black,fill=cffffd7,line join=miter,line cap=butt,even odd rule,line
    width=0.800pt] (67.8823,382.0250) .. controls (49.9030,364.6868) and
    (20.3354,369.7089) .. (-7.7782,372.8326) .. controls (-35.5091,375.9138) and
    (-66.9549,381.0977) .. (-86.9741,401.1169) .. controls (-118.5186,432.6613)
    and (-107.3014,547.7562) .. (-84.8528,581.4291) .. controls
    (-77.5110,592.4418) and (-45.0658,606.4766) .. (-33.2340,610.4205) .. controls
    (-3.9645,620.1770) and (46.8891,584.9646) .. (67.1751,580.9646);
   \draw[color=black,fill=white,pattern=dots, pattern color=green!20 ] (0,480) circle (40);
   \draw (0,480) node {$B_0$};
   \draw (0,400) node {$B$};
   \draw (320,525) node {$N$};
   \draw (320,590) node {$Z$};
\end{tikzpicture}}
\caption{The final step of constructing the cobordisms $W_{01}$, $W_{12}$ and $W_{02}$. The shaded part is $W_{01}$. The
unfilled part is $W_{12}$.}\label{fig:conc3}
\end{figure}
\begin{itemize}
\item[\textbf{Step 1.}] Begin with $A\subset S^3\times[1,2]$ and blow up all the double points of $A$ (we do not specify yet, whether we
perform positive or negative blow-ups). The exceptional divisors are denoted by $E_1,\ldots,E_p$. For a component $A_j$ of $A$, let $\wt{A}_j$
be its strict transform. Set
\begin{equation}\label{eq:dij} d_{ij}:=A_i\cdot E_j\textrm{ and } \wt{d}_{ij}:=\frac{A_i\cdot E_j}{E_j\cdot E_j}.\end{equation}

\item[\textbf{Step 2.}] Fix a framing vector $\qq_2=(q_{21},\ldots,q_{2n_2})\in\Z^{n_2}$, 
where $n_2$ is the number of components of the link $\LL_2$. This turns $\LL_2$
into a framed link and let $\Lambda_2$ be its framing matrix. Attach to the $p$-fold blow-up of $S^3\times[1,2]$ constructed
in Step~1 $n_2$ two-handles along $\LL_2$ with framings given by $\qq_2$. The resulting manifold is called $Z$.
Its boundary is $Y_2\sqcup -S^3\times\{1\}$, where $Y_2=S^3_{\qq_2}(\LL_2)$ is a surgery on $\LL_2$; see Figure~\ref{fig:conc2}. 
Let $C_1,\ldots,C_{n_2}$ be the cores of
the handles attached. For each component $L_{2i}$ of $\LL_2$ choose a Seifert surface $\Sigma_{2i}\subset S^3\times\{2\}$. Let 
$F_{2i}=\Sigma_{2i}\cup C_i$. Then $F_{21},\ldots,F_{2n_2}$ are closed connected surfaces. 
\item[\textbf{Step 3.}] 
Form the union $D=C_1\cup\ldots\cup C_{n_2}\cup\wt{A}_1\cup\ldots\cup\wt{A}_n$. Denote by $D_1,\ldots,D_n$ the connected components of $D$
as in Figure~\ref{fig:conc2}.
Let $N$ be a tubular neighbourhood of $D$ in $Z$; see Figure~\ref{fig:conc3}. 
Set $W_{12}=Z\setminus N$ and let $Y_1=\partial W\setminus Y_2$. Then
$W_{12}$ is a cobordism between $Y_1$ and $Y_2$.
\item[\textbf{Step 4.}] Take again $Z$ and glue to it a four-ball $B$ along $S^3\times\{1\}\subset\partial Z$. Let $X_{\qq_2}$
be the resulting manifold. Finally pick a small ball $B_0\subset B$ and drill it out from $X_{\qq_2}$. Set $W_{02}=X_{\qq_2}\setminus B_0$
and $W_{01}=(B\setminus B_0)\cup N$, so that $W_{01}\cup W_{12}=W_{02}$. 
Let $Y_0=S^3$. Then $W_{01}$ is a cobordism between $Y_0$ and $Y_1$ and $W_{02}$ is a cobordism between $Y_0$ and $Y_2$. See Figure~\ref{fig:conc3}.
\end{itemize}

We have the following immediate observation.
\begin{lemma}\
\begin{itemize}
\item[(a)] The cobordism $W_{02}$ is a $p$-fold blow-up of the cobordism $U_{\qq_2}(\LL_2)$ defined before Proposition~\ref{degreeshift}.  
\item[(b)] Suppose $A$ is a SPSIC. Then $D_1,\ldots,D_n$ are disks and $n=n_1$ is the number of components of $\LL_1$. Furthermore
$Y_1$ is a surgery on $\LL_1$ with a framing vector $\qq_1=(q_{11},\ldots,q_{1n_1})$ 
depending on $\qq_2$ and the signs of blow-ups (we give a precise formula for $\qq_1$ below). The cobordism $W_{01}$
is identified with $U_{\qq_1}(\LL_1)$.
\end{itemize}
\end{lemma}

From now on we will assume that $A$ is a SPSIC.

Choose Seifert surfaces for $L_{11},\ldots,L_{1n_1}$ and call them $\Sigma_{11},\ldots,\Sigma_{1n_1}$. Let $F_{11},\ldots,F_{1n_1}$ be
closed surfaces obtained by
capping the disks $D_1,\ldots,D_{n_1}$ with $\Sigma_{11},\ldots,\Sigma_{1n_1}$.
The classes $[F_{11}],\ldots,[F_{1n_1}]$ generate $H_2(W_{01};\Z)$. The map $W_{01}\hookrightarrow W_{02}$ induces a monomorphism
on second homologies. We will not distinguish between the class $[F_{1i}]\in H_2(W_{01};\Z)$ and its image in $H_2(W_{02};\Z)$.

\begin{lemma}\label{lem:relationofclasses}
In $H_2(W_{02},\Z)$ we have the following relation.
\[ [F_{1i}]=\sum_{j\in \Theta_i} [F_{2j}]+\sum_{k=1}^p \wt{d}_{ik}[E_k].\]
\end{lemma}
\begin{proof}
The class of $[F_{1i}]$ is the sum of a class $\sum_{j\in\Theta_i}[F_{2j}]$ and the class of an immersed sphere 
\[S_{i}=\Sigma_{1i}\cup\wt{A}_i\cup\sum_{j\in\Theta_i}\Sigma_{2j}.\]
The spheres $S_i$ will usually be only immersed, because the Seifert surfaces $\Sigma_{21},\ldots,\Sigma_{2n_2}$ can intersect (we may assume that
their intersection is transverse, but this is relevant for the present proof).

Notice that $S_{i}$ can be regarded as a strict transform of a closed surface in $S^3\times[0,1]$, formed by capping the component $A_i$
with the Seifert surfaces of corresponding links. This surface in $S^3\times[0,1]$ is homologically trivial, as $H_2(S^3\times[0,1];\Z)=0$. 
Therefore, the class of $S_i$
in $H_2(W_{02};\Z)$ is a linear combination of classes generated by the exceptional divisors. The coefficients in this linear
combinations can be calculated by intersecting $S_i$ with divisors $E_1,\ldots,E_p$. More concretely $[S_i]=\sum\frac{S_i\cdot E_k}{E_k\cdot E_k}[E_k]$.
But geometrically $S_i\cdot E_k=A_i\cdot E_k=d_{ik}$. The lemma follows.
\end{proof}

\begin{lemma}\label{lem:calculateqq1}
If $A$ is a SPSIC and all the blow-ups are of fixed sign (either all positive or all negative), then $\qq_1$ and $\qq_2$
are related by the following formula.
\begin{align}
\intertext{If all the blow-ups are negative}
q_{1i}&=\sum_{j\in \Theta_i} q_{2j}-4\newm_{ii}-a_i.\label{eq:pos1}\\
\intertext{If all the blow-ups are positive}
q_{1i}&=\sum_{j\in \Theta_i} q_{2j}+a_i\label{eq:pos2}.
\end{align}
\end{lemma}
\begin{proof}
The coefficients $q_{11},\ldots,q_{1n_1}$ are self-intersections of disks $D_1,\ldots,D_{n_1}$. Here, by the word 'self-intersection' we mean the
following: push slightly $D_i$ to obtain another disk, 
called $D_i'$, intersecting $D_i$ transversally and such that $\partial D_i'\subset S^3\times\{1\}$ is disjoint
from $\partial D_i$ and the linking number 
$\lk(\partial D_i,\partial D_i')$ calculated on $S^3\times\{1\}$ is equal to zero. Then the self-intersection of $D_i$ is defined
as the number of intersection points (counted with signs) of $D_i$ and $D_i'$. In other words, the self-intersection of $D_i$ is
equal to the self-intersection of $[F_{1i}]$. On the other hand, the framing $q_{2j}$ is interpreted in the same way as the self-intersection
of $[F_{2j}]$. 

As the classes $[F_{2j}]$ and $[E_k]$ are orthogonal, by Lemma~\ref{lem:relationofclasses} the difference of self-intersections
\[\left(\sum_{j\in\Theta_i}[F_{2j}]\right)^2-[F_{1i}]^2\]
is equal to
\[\left(\sum_{k=1}^p\wt{d}_{ik}[E_k]\right)^2=\sum_{k=1}^p\wt{d}_{ik}^2[E_k]^2.\]

Now we have two cases. First suppose that all the blow-ups are negative. Then $[E_k]^2=-1$ for all $k$. Moreover, $\wt{d}_{ik}^2=0,1$ or $4$ is the square of
the multiplicity of the double point $z_k$ if $z_k\in A_i$ and $\wt{d}_{ik}=0$ if $z_k\notin A_i$.
As $a_i$ is equal to the number of multicolored double points on $A_i$, while $\newm_{ii}$ is the number of monochromatic double points on $A_i$.
This proves \eqref{eq:pos1}.

The situation with positive blow-ups is analogous. There is one difference, though. If an exceptional divisor $E_k$ is a blow-up of a monochromatic point
on $A_i$, then $d_{ik}=0$ (and not $\pm 2$). 
This corresponds to the fact that in the blow-up the annulus
$\wt{A_i}$ will intersect the exceptional divisor $E_k$ in two points with opposite orientations.
\end{proof}

\subsection{Homological properties of $W_{12}$.}

We will need to study some homological properties of $W_{02}$ and $W_{12}$. They are synthesized in the following lemma,
which is a direct generalization of results in \cite[Section 2.2]{BL2}.
\begin{lemma} \label{lem:definite} Suppose $A$ is sprouting and $\qq_2$ has all coordinates sufficiently large.
\begin{itemize}
\item We have $H_1(W_{12};\Q)=0$ and $H_2(W_{12};\Q)\cong\Q^{p+n_2-n_1}$. Moreover $H_2(W_{12};\Z)$ is the coimage of the map 
$H_2(Z;\Z)\to\Z^{n_1}$ which takes $x$ to the vector $(x\cdot D_1,\ldots,x\cdot D_{n_1})$.
\item If all the blow-ups are positive, the manifold $W_{12}$ has positive definite intersection form on $H_2$. If all the blow-ups
are negative definite and $A$ is an APSIC, then $W_{12}$ has negative definite intersection form.
\end{itemize}
\end{lemma}
\begin{proof}
The manifold $W_{12}$ is built from $Z$ by removing tubular neighborhoods of disks. As $Z$ arises by gluing $n_2$
two-handles to the $p$-fold blow-up of $S^3\times[0,1]$ and the framing matrix is nondegenerate, we have $H_1(Z;\Q)=0$, $H_2(Z;\Q)\cong\Q^{p+n_2}$
and $H_3(Z;\Q)=0$, furthermore $H_3(Z;\Z)=0$.

Consider the long exact sequence of homology (with $\Z$ coefficients) of the pair $(Z,W_{12})$. By excision we have 
$H_*(Z,W_{12})\cong H_*(N,\partial_+N)$. Here $N$ is viewed as a $D^2$ bundle over $D$ and $N_+$ is the associated
$S^1$ bundle. Using Thom isomorphism we obtain that $H_3(N,\partial_+N)=0$ and that 
$H_2(N,\partial_+N)\cong\Z^{n_1}$ is generated by classes $\alpha_j:=[({t_j}\times D^2,{t_j}\times S^1)]\in
H_2(N,\partial_+ N;\Z)$, $j=1,\ldots,n_1$, where
$t_1,\ldots,t_{n_1}$ are some points in $D_1,\ldots,D_{n_1}$ respectively.  

The latter implies that $H_2(W_{12};\Z)$ injects into $H_2(Z;\Z)$. The map 
\[\kappa\colon H_2(Z;\Z)\to H_2(Z,W_{12};\Z)\]
can be explicitly described. Namely, for $x\in H_2(Z;\Z)$ we choose its representative as a union of cycles
each intersecting $D$ transversally. Then 
\[\kappa(x)=(x\cdot D_1)\alpha_1+(x\cdot D_2)\alpha_2+\ldots+(x\cdot D_{n_1})\alpha_{n_1}.\]
This proves the first part of the lemma.

The first part can also be rephrased in another way. As each class $x\in H_2(Z;\Z)$ can be represented by a surface disjoint
from $S^3\times\{1\}$, the geometric intersection number $(x\cdot D_j)$ is equal to $x\cdot F_{1j}$.  
With  this description it follows $H_2(W_{12};\Z)$
is an orthogonal (with respect to the intersection form) complement to a submodule of $H_2(Z;\Z)$ generated by 
$[F_{11}],\ldots,[F_{1n_1}]$. The same applies for homologies with $\Q$ coefficients. Therefore, the signature of the intersection
form on $W_{12}$ can be calculated as the difference of the signature of the intersection form on $Z$ and  the signature of
the intersection form on $[F_{11}],\ldots,[F_{1n_1}]$. The proof of the second part follows now by a case by case analysis.

If all the blow-ups are positive, then $Z$ has a positive definite intersection form, hence it restricts to a positive
definite intersection form on $W_{12}$. If $A$ is an APSIC and  all the blow-ups are negative, then one readily computes that 
$b_2^+(Z)=n$ and $b_2^-(Z)=p$. Moreover, the intersection form on an $n$-dimensional subspace spanned by
$[F_{11}],\ldots,[F_{1n}]$ is positive definite. So its orthogonal complement is negative definite.

\end{proof}

\section{Inequalities for the \HFF function under the crossing change}
We will now assume that links $\LL_1$ and $\LL_2$ are connected by a PSIC. The inequality for $d$-invariants \eqref{eq:dinvariantinequality}
will translate into the inequality between \HFF functions, or, equivalently, \JFF functions. 

We are going to prove the following two results.
\begin{theorem}\label{thm:mainestimate}
Let $A$ be an APSIC from $\LL_1$ to $\LL_2$. Let $J_1$ and $J_2$ be the \JFF functions as in Definition~\ref{def:J}.
Set $\rr=(\newm_{11},\ldots,\newm_{nn})$. Choose a presentation of $\newm_{il}$ for $i<l$
\[\newm_{il}=\newm_{il}^1+\newm_{il}^2,\]
where $\newm_{il}^1, \newm_{il}^2$ are non-negative integers. Set 
\[k_i=\sum_{j<i} \newm_{ji}^2+\sum_{j>i}\newm_{ij}^1.\]
Let $\kk=(k_1,\ldots,k_n)$. Then for any $\mm\in\Z$ we have
\begin{equation}\label{eq:mainestimate}
J_1(\mm+\kk)\le J_2(\mm)\le J_1(\mm-\rr)
\end{equation}
\end{theorem}

A counterpart of this result for an EPSIC is the following.
\begin{theorem}\label{thm:epsic}
Suppose $A$ is an EPSIC. Choose $\mm_2\in\Z^{n_2}$ and let $\mm_1\in\Z^{n_1}$ be given by 
 $\mm_{1i}=\mm_{2i}$ if $i\neq k$ and $\mm_{1k}=\mm_{2k}+\mm_{2,n_2}$. Then
\begin{equation}\label{eq:sprj}
J_1(\mm_1)\le J_2(\mm_2).
\end{equation}
\end{theorem}

Theorem~\ref{thm:mainestimate} is proved in Sections~\ref{sec:positiveblowup} and \ref{sec:negativeblowup}. Theorem~\ref{thm:epsic}
is proved in Section~\ref{sec:epsic}. In Section~\ref{sec:single} we prove Theorem~\ref{thm:single}, which is a straightforward, but
important, corollary of Theorem~\ref{thm:mainestimate}.

\subsection{Proof of Theorem~\ref{thm:mainestimate}. Part 1.}\label{sec:negativeblowup}
In this section we prove the part $J_2(\mm)\le J_1(\mm-\rr)$.

\smallskip
Construct $W_{02}$  by making negative blow-ups of the APSIC concordance; see Section~\ref{sec:topoconstruct}. Choose $\mm\in\Z^n$. Pick
$\qq_2$ sufficiently large (we specify
below the precise meaning of sufficiently large), but now we point out that $\qq_2$ is chosen \emph{after} $\mm$.
According to Lemma~\ref{lem:calculateqq1} choose $\qq_1=\qq_2-4\rr-\aaa$.

Set $\vv_2=\mm-\lkv_2$, where $\lkv_2$ is the linking vector for $\LL_2$. 
Let $\sss_{\vv_2}$ be the \spinc{} structure on $Y_2$.
It extends to a \spinc{} structure $\sst_{\vv_2}$ on $U_{\qq_2}(\LL_2)$ (see Section~\ref{sec:surgeryonlinks} for definition of $U_{\qq}(\LL)$). 
Recall that $W_{02}$ is identified with a $p$-fold
blow-up of $U_{\qq_2}(\LL_2)$. Let $\pi\colon W_{02}\to U_{\qq_2}(\LL_2)$ be the blow-down map. Define
the \spinc{} structure $\sst'_{\vv_2}$ on $W_{02}$ as the \spinc{} structure, whose first Chern class is equal to
\[\pi^*c_1(\sst_{\vv_2})+PD[E_1]+\ldots+PD[E_n].\]

\begin{lemma}\label{lem:spincrestricts}
The \spinc{} structure $\sst'_{\vv_2}$ on $W_{02}$ restricts to the \spinc{} structure $\sss_{\vv_1}$
on $Y_1$, where 
\begin{equation}\label{eq:spincrestricts}
\vv_1=\vv_2-\rr-\frac12\aaa.
\end{equation}
\end{lemma}
\begin{proof}[Proof of Lemma~\ref{lem:spincrestricts}]
By construction of $W_{01}$ and by Lemma~\ref{lem:calculateqq1} we have $W_{01}=U_{\qq_1}(\LL_1)$. 
The \spinc{} structure $\sss_{\vv_1}$ on $Y_1$ extends to the \spinc{} structure $\sst_{\vv_1}$ on $W_{01}$. Our aim is
to show that with the choice of $\vv_1$ as in the statement of the lemma, $c_1(\sst_{\vv_1})$ and $c_1(\sst_{\vv'_2})$ evaluate
in the same way on the classes $[F_{11}],\ldots,[F_{1n}]$.

By definition of $\sst'_{\vv_2}$ we have
\begin{equation}\label{eq:newchernevaluates1}
\begin{split}
\langle c_1(\sst'_{\vv_2}),[F_{2i}]\rangle&=2v_{2,i}-\left(\Lambda_{21}+\ldots+\Lambda_{2n_2}\right)_i\\
\langle c_1(\sst'_{\vv_2}),[E_j]\rangle&=E_j\cdot E_j=-1,
\end{split}
\end{equation}
where $\Lambda_{21},\ldots,\Lambda_{2n_2}$ are column vectors of the 
framing matrix $\Lambda_2$ for $\LL_2$.
The subscript $i$ in the first formula means that we take the $i$-th coordinate of the vector in the parentheses.

Combining \eqref{eq:newchernevaluates1} with Lemma~\ref{lem:relationofclasses} we obtain
\begin{equation}\label{eq:secondchernevaluates1}
\langle c_1(\sst'_{\vv_2}),[F_{1i}]\rangle=2v_{2,i}-\left(\Lambda_{21}+\ldots+\Lambda_{2n}\right)_i+\sum_{j=1}^p d_{ij}.
\end{equation}
The framing matrices $\Lambda_2$ and $\Lambda_1$  can be compared using \eqref{eq:linkingchange11} and \eqref{eq:pos1}.
\[
\left(\Lambda_{21}+\ldots+\Lambda_{2n}\right)_i-\left(\Lambda_{11}+\ldots+\Lambda_{1n}\right)_i\stackrel{\eqref{eq:linkingchange11}}{=}
\sum_{j\neq i} \newm_{ij}+q_{2i}-q_{1i}\stackrel{\eqref{eq:pos1}}{=}4\newm_{ii}+2a_i.
\]
Notice that $d_{ij}=1$ for all multicolored double points that lie on $A_i$ and $d_{ij}=2$
for all monochromatic double points on $A_i$. Therefore \eqref{eq:secondchernevaluates1}
implies that

\[\langle c_1(\sst'_{\vv_2}),[F_{1i}]\rangle=2v_{2,i}-\left(\Lambda_{21}+\ldots+\Lambda_{2n}\right)_i+a_i+2\newm_{ii}.\]
The two above equations yield
\[\langle c_1(\sst'_{\vv_2}),[F_{1i}]\rangle=2v_{2,i}-\left(\Lambda_{11}+\ldots+\Lambda_{1n}\right)_i-a_i-2\newm_{ii}.\]
On the other hand, by Proposition~\ref{prop:mo125}
\begin{equation}\label{eq:c1f11}
\langle c_1(\sst_{\vv_1}),[F_{1i}]\rangle=2v_{1,i}-\left(\Lambda_{11}+\ldots+\Lambda_{1n}\right)_i.
\end{equation}
Combining the two above formulae we conclude that $c_1(\sst_{\vv_1})$ and $c_1(\sst'_{\vv_2})$ evaluate to the
same number on each of the $[F_{1i}]$. 
\end{proof}

We resume the proof of inequality $J_2(\mm)\le J_1(\mm-\rr)$.
If $\qq_2$ is large, then the statement of Theorem~\ref{thm:surgeryonlinks} holds for $\qq_2$ surgery on $\LL_2$
and for $\qq_1$ surgery on $\LL_1$.
Furthermore, we require that $\qq_2$ and $\qq_1$ are large enough so that $\vv_2\in P_\H(\Lambda_2)$
and $\vv_1\in P_\H(\Lambda_1)$.

Denote by $\sst$ the restriction of $\sst'_{\vv_2}$ to $W_{12}$. By Lemma~\ref{lem:spincrestricts}, $(W_{12},\sst)$
is a \spinc{} cobordism between $(Y_1,\sss_{\vv_1})$ and $(Y_2,\sss_{\vv_2})$.

By Lemma~\ref{lem:definite} $W_{12}$ is negative definite. By \eqref{eq:degreedefinition} we have:
\begin{equation}\label{eq:dw12}
d(Y_2,\sss_{\vv_2})\ge \deg F_{(W_{12},\sst)}+d(Y_1,\sss_{\vv_2}).
\end{equation}
By Proposition~\ref{prop:blowup} 
combined with Theorem~\ref{thm:surgeryonlinks}: 
\begin{equation}\label{eq:disnextstep}
\begin{split}
d(Y_2,\sss_{\vv_2})&=-\deg F(-W_{02},\sst)-2J_2(\vv_2-\lkv_2)-p\\
d(Y_1)&=-\deg F(-W_{01},\sst)-2J_1(\vv_1-\lkv_1).
\end{split}
\end{equation}
Notice that the first equation contains the term $-p$. 
This follows from the fact that $W_{02}$ is not $U_{\qq_2}(\LL_2)$, but it is a negative blow-up 
of $U_{\qq_2}(\LL_2)$ with $p$ blow-ups. If we reverse the orientation, the negative blow-up becomes a positive blow-up, so $\deg F_{(-W_{02},\sst)}=
\deg F_{(U'_{\qq_2}(\LL_2),\sst)}-p$ by Proposition~\ref{prop:blowup}.

Substituting \eqref{eq:disnextstep} into \eqref{eq:dw12} we obtain:
\begin{equation}\label{eq:intermediatestep}
-\deg F_{(-W_{02},\sst)}+\deg F_{(-W_{01},\sst)}-\deg F_{(W_{12},\sst)}-p+J_1(\vv_1-\lkv_1)\ge J_2(\vv_2-\lkv_2).
\end{equation}
Let us look at the expression 
\[\Delta:=\deg F_{(-W_{02},\sst)}+\deg F_{(W_{12},\sst)}-\deg F_{(-W_{01},\sst)}.\] 
Denote by $c_{02},c_{12}$ and $c_{01}$ the integrals of $c_1^2(\sst)$ over $W_{02}$, $W_{12}$ and $W_{01}$ respectively. Likewise
denote by $\sigma_{02},\sigma_{12}$ and $\sigma_{01}$ the corresponding signatures and $\chi_{02}$, $\chi_{12}$, $\chi_{01}$ the Euler
characteristic. We have by \eqref{eq:degreedefinition}:
\begin{align*}
4\deg F_{(-W_{02},\sst)}&=-c_{02}+3\sigma_{02}-2\chi_{02}\\
4\deg F_{(W_{12},\sst)}&=c_{12}-3\sigma_{12}-2\chi_{12}\\
-4\deg F_{(-W_{01},\sst)}&=c_{01}-3\sigma_{01}+2\chi_{01}.
\end{align*}
Notice that in the above expression we switched signs of $\sigma$ and $c$ according to the orientation. Notice also that 
$\sigma_{02}=\sigma_{01}+\sigma_{12}$ and $\chi_{02}=\chi_{01}+\chi_{12}$ (additivity of the signature and of the Euler characteristic)
and $c_{02}=c_{01}+c_{12}$ (functoriality of the Chern class). 
Summing up the three equations we obtain
\[4\Delta=4\chi_{12}=-4\chi(W_{12}).\]
The Euler characteristic of $W_{12}$ can be quickly calculated. Recall that in Section~\ref{sec:topoconstruct} the manifold $W_{12}$
was constructed by taking $S^3\times[0,1]$, blowing up $p$ times, gluing $n$ two-handles and drilling out $n$ disks. The original $S^3\times[0,1]$
has Euler characteristic $0$. Each blow-up increases it by $1$. A two-handle attachment increases it by $1$ and drilling out a disk decreases
it by $1$. Finally $\chi(W_{12})=p$ so $\Delta=-p$. Plugging the value of $\Delta$ into \eqref{eq:intermediatestep} we obtain.
\[J_1(\vv_1-\lkv_1)\ge J_2(\vv_2-\lkv_2).\] 

\smallskip
By definition, $\vv_2-\lkv_2=\mm$. The last step is to calculate $\vv_1-\lkv_1$. We use Lemma~\ref{lem:spincrestricts}. By Lemma~\ref{lem:lkv},
equation \eqref{eq:spincrestricts} can be rewritten as
\[\vv_1=\vv_2-\kk-(\lkv_2-\lkv_1).\] 
This amounts to saying that $\vv_1-\lkv_1=\mm-\rr$, so $J_1(\mm-\rr)\ge J_2(\mm)$.

\subsection{Proof of Theorem~\ref{thm:mainestimate}. Part 2.}\label{sec:positiveblowup}
We are going to prove the part $J_1(\mm+\kk)\le J_2(\mm)$.

The proof is analogous, although there are some differences. We construct $W_{02}$ by making all blow-ups positive. Choose $\mm\in\Z^n$
and let $\qq_2$ be sufficiently large.

We begin with some combinatorics. 
The exceptional divisors of the blow-up are $E_1,\ldots,E_p$. We choose orientation of the divisors by requiring that if $E_j$ is
the exceptional divisor of the blow-up of the point of intersection $A_i\cap A_{i'}$ with $i<i'$, then $E_j\cap A_i=1$ and $E_j\cap A_{i'}=-1$.
The orientation of the exceptional divisors of blow-ups of monochromatic double points is relevant.

Choose now $\delta_1,\ldots,\delta_p\in\{-1,+1\}$ in the following way.
If $E_j$ is the exceptional divisor of the blow-up of a monochromatic double point, then $\delta_j=1$. 
Let now $i$ and $i'$ be the indices such that
$i<i'$. Let $I_{ii'}$ be the set of indices $\{1,\ldots,p\}$ such that if $j\in I_{ii'}$, then $E_j$
is the exceptional divisor of the blow-up of a point in $A_i\cap A_{i'}$. We know that $\# I_{ii'}=\newm_{ii'}$. Partition
the set $I_{ii'}$ into two subsets $I^1_{ii'}$ and $I^2_{ii'}$ of cardinality $m^1_{ii'}$ and $m^2_{ii'}$ respectively.
Set $\delta_j=-1$ for $j\in I^1_{ii'}$ and $\delta_j=1$ for $j\in I^2_{ii'}$. Finally denote
\[\ddl=(\theta_1,\ldots,\theta_n)=\left(\sum_{l=1}^p\delta_l d_{1l},\ldots,\sum_{l=1}^p\delta_l d_{nl}\right).\]
We have the following result
\begin{lemma}\label{lem:newdiff}
With the choice of $\delta_1,\ldots,\delta_p$ as above and with $\kk$ as in the statement of Theorem~\ref{thm:mainestimate} we have
\[\lkv_2-\lkv_1-\kk=\frac12\ddl.\]
\end{lemma}
\begin{proof}
In view of Lemma~\ref{lem:lkv} we need to prove that 
\[\ddl=2\kk-\aaa.\]
By definition, $k_i=\sum_{j<i} \newm_{ji}^2+\sum_{j>i} \newm_{ij}^1$. Using the definition of $a_i$ in \eqref{eq:defai}
and the fact that for $i<j$ $\newm_{ji}=\newm_{ij}=\newm_{ij}^1+\newm_{ij}^2$
we transform the above equation into the following set of equations for $i=1,\ldots,n$:
\begin{equation}\label{eq:deltaeq}
\sum_{l=1}^p\delta_l d_{il}=\sum_{j<i} (\newm_{ji}^2-\newm_{ji}^1)+\sum_{i'>i} (\newm_{ji}^1-\newm_{ji}^2).
\end{equation}
The way the exceptional divisors are oriented implies that
$d_{il}=1$ if $l\in I_{ii'}$ for some $i'>i$, $d_{il}=-1$ if $l\in I_{ii'}$ for some $i'<i$, otherwise $d_{il}=0$.
The left hand side of \eqref{eq:deltaeq} can be expressed as
\[\sum_{l=1}^p\delta_l d_{il}=\sum_{i'<i}\sum_{l\in I_{ii'}} \delta_l-\sum_{i'>i}\sum_{l\in I_{ii'}}\delta_l.\]
But $\sum_{l\in I_{ii'}}\delta_l=\newm_{i'i}^2-\newm_{i'i}^1$ by definition, so
\[\sum_{l=1}^p\delta_l d_{il}=\sum_{j<i} (\newm_{ji}^2-\newm_{ji}^1)+\sum_{i'>i} (\newm_{ji}^1-\newm_{ji}^2).\]
This proves \eqref{eq:deltaeq} and concludes the proof of the lemma.
\end{proof}

We resume the proof of Theorem~\ref{thm:mainestimate}. 
The manifold $W_{02}$ is a $p$-fold positive blow-up of $U_{\qq_2}(\LL_2)$,
and let again $\pi$ be the blow-down map. Choose $\vv_2=\mm+\lkv_2$ and the \spinc{} structure $\sst'_{\vv_2}$ on $W_{02}$ 
given by
\[c_1(\sst'_{\vv_2})=\pi^*c_1(\sst_{\vv_2})+\delta_1 PD[E_1]+\ldots+\delta_n PD[E_n].\]
We have the following result, which is a counterpart of Lemma~\ref{lem:spincrestricts}.
\begin{lemma}\label{lem:newdiff2}
The \spinc{} structure $\sst'_{\vv_2}$ restricts to the \spinc{} structure $\vv_1$ on $Y_1$, where
\[\vv_1=\vv_2-\frac12\ddl.\]
\end{lemma}
\begin{proof}[Proof of Lemma~\ref{lem:newdiff2}]
By definition of $\sst'_{\vv_2}$ we obtain
\begin{equation}\label{eq:newchernevaluates2}
\begin{split}
\langle c_1(\sst'_{\vv_2}),[F_{2i}]\rangle&=2v_{2,i}-\left(\Lambda_{21}+\ldots+\Lambda_{2n_2}\right)_i\\
\langle c_1(\sst'_{\vv_2}),[E_j]\rangle&=\delta_jE_j\cdot E_j=\delta_j.
\end{split}
\end{equation}

Combining \eqref{eq:newchernevaluates2} with Lemma~\ref{lem:relationofclasses} we obtain
\begin{equation}\label{eq:secondchernevaluates2}
\langle c_1(\sst'_{\vv_2}),[F_{1i}]\rangle=2v_{2,i}-\left(\Lambda_{21}+\ldots+\Lambda_{2n}\right)_i+\sum_{j=1}^p \delta_{j}d_{ij}.
\end{equation}

Notice that by Lemma~\ref{lem:calculateqq1} $Y_1$ is a $\qq_1$ surgery on $\LL_1$, where
$\qq_1=\qq_2+\aaa$.
Therefore a quick calculation using \eqref{eq:linkingchange11} yields
\begin{multline*}
\left(\Lambda_{21}+\ldots+\Lambda_{2n}\right)_i-\left(\Lambda_{11}+\ldots+\Lambda_{1n}\right)_i=\\
\left(\Lambda_{21}+\ldots+\Lambda_{2,i-1}+\Lambda_{2,i+1}+\ldots+\Lambda_{2n}\right)_i\\-
\left(\Lambda_{21}+\ldots+\Lambda_{2,i-1}+\Lambda_{2,i+1}+\ldots+\Lambda_{2n}\right)_i\\+q_{2,i}-q_{1,i}=
a_i-a_i=0.
\end{multline*}
Substituting this into \eqref{eq:secondchernevaluates2} 
we obtain.
\[
\langle c_1(\sst'_{\vv_2}),[F_{1i}]\rangle=2v_{2,i}-\left(\Lambda_{11}+\ldots+\Lambda_{1n}\right)_i+\sum_{j=1}^p \delta_{j}d_{ij}.
\]
The evaluation of $c_1(\sst_{\vv_1})$ on $[F_{1i}]$ is given by \eqref{eq:c1f11}. We obtain that
\[\langle c_1(\sst'_{\vv_2}),[F_{1i}]\rangle=\langle c_1(\sst_{\vv1}),[F_{1i}]\rangle, \textrm{ if }2\vv_2-\ddl=2\vv_1.\]
This is exactly the statement of the lemma.
\end{proof}

Resuming the proof of Theorem~\ref{thm:mainestimate} we obtain that with the choice of $\vv_1$ as in Lemma~\ref{lem:newdiff2},
the manifold $(W_{12},\sst'_{\vv_2})$ is a \spinc{} cobordism between $(Y_1,\sss_{\vv_1})$ and $(Y_2,\sss_{\vv_2})$.
By Lemma~\ref{lem:definite} $W_{12}$ is positive definite. Then $-W_{12}$ is negative definite and
\eqref{eq:dinvariantinequality} gives.
\[d(Y_1,\sss_{\vv_1})\ge \deg F_{(-W_{12},\sst)}+d(Y_2,\sss_{\vv_2}).\]
Plugging again the formula for $d$--invariants of large surgeries we obtain.
\begin{equation}\label{eq:degreescancelout}
-\deg F_{(-W_{01},\sst)}+\deg F_{(-W_{02},\sst)}-\deg F_{(-W_{12},\sst)}+J_2(\vv_2-\lkv_2)\ge J_1(\vv_1-\lkv_1).
\end{equation}
Now the expression $\deg F_{(-W_{01},\sst)}+\deg F_{(-W_{12},\sst)}-\deg F_{(-W_{02},\sst)}$ is much easier to handle
than an analogous expression in Section~\ref{sec:negativeblowup}  because 
$-W_{02}=-W_{01}\cup -W_{12}$. Therefore the map $F_{(-W_{02},\sst)}$ is the composition of $F_{(-W_{01},\sst)}$ and $F_{(-W_{12},\sst)}$ so
its degree is the sum of the degrees of the summands.
The three degrees in \eqref{eq:degreescancelout} cancel out and we are left with
\begin{equation}\label{eq:penultimate}
J_2(\vv_2-\lkv_2)\ge J_1(\vv_1-\lkv_1).
\end{equation}
By definition $\vv_2-\lkv_2=\mm$. On the other hand, by Lemma~\ref{lem:newdiff} combined with Lemma~\ref{lem:newdiff2}:
\[\lkv_2-\lkv_1-\kk=\frac12\ddl=\vv_2-\vv_1.\]
Plugging this into \eqref{eq:penultimate} yields $J_2(\mm)\ge J_1(\mm+\kk)$. This accomplishes the proof of Theorem~\ref{thm:mainestimate}.

\subsection{Proof of Theorem~\ref{thm:epsic}}\label{sec:epsic}

The construction is similar as in Section~\ref{sec:positiveblowup}. Take $\mm_2\in\Z^{n_2}$ and let $\qq_2$ be sufficiently large.
The construction of $W_{02}$ is as as in the proof of Theorem~\ref{thm:mainestimate}, but there are no blow-ups, hence $W_{02}=U_{\qq_2}(\LL_2)$. We know that
$W_{01}=U_{\qq_1}(\LL_1)$, where by Lemma~\ref{lem:calculateqq1} $q_{1i}=q_{2i}$ if $i\neq k$ and $q_{1k}=q_{2k}+q_{2n_2}$.
Set $\vv_2=\mm_2-\lkv_2$ and let $\sst_{\vv_2}$ be the \spinc{} structure on $W_{02}$ extending the \spinc{} structure
$\sss_{\vv_2}$ on $Y_2$. Evaluating $c_1(\sst_{\vv_2})$ on classes $[F_{11}],\ldots,[F_{1n_1}]$ we show
that $\sst_{\vv_2}$ restricts to $\sss_{\vv_1}$ on $Y_1$, where $v_{1i}=v_{2i}$ if $i\neq k$ and $v_{1k}=v_{2k}+v_{2n_2}$.
By Lemma~\ref{lem:definite}, $W_{12}$ is positive definite. Therefore $(-W_{12},\sst_{\vv_2})$ is a negative
definite \spinc{} cobordism between $(Y_2,\sss_{\vv_2})$ and $(Y_1,\sss_{\vv_1})$. Acting exactly in the same way as in 
Section~\ref{sec:positiveblowup} we arrive at the inequality $J_2(\vv_2-\lkv_2)\ge J_1(\vv_1-\lkv_1)$. We have $\vv_2-\lkv_2=\mm_2$.
Moreover it is easy to see that with the definition of $\vv_1$ and $\mm_1$, we have $\vv_1-\lkv_1=\mm_1$. This concludes the proof.

\subsection{A variant of Theorem~\ref{thm:mainestimate} for a single crossing change}\label{sec:single}
\begin{theorem}\label{thm:single}
Let $\LL_1$ and $\LL_2$ be two $n$-component links differing by a single positive
crossing change, that is, $\LL_2$ arises by changing a negative crossing of $\LL_1$
into a positive one. Let $J_1$ and $J_2$ be the corresponding \JFF functions and let $\mm\in\Z^n$, $\mm=(m_1,\ldots,m_n)$.
\begin{itemize}
\item[(a)] If the crossing change is between two strands of the same component $L_{1i}$, then
\[J_2(m_1,m_2,\ldots,m_i+1,\ldots,m_n)\le J_1(m_1,\ldots,m_n)\le J_2(m_1,\ldots,m_i,\ldots,m_n).\]
\item[(b)] If the crossing change is between the $i$-th and $j$-th strand of $\LL_1$, then
\[J_2(m_1,m_2,\ldots,m_n)\le J_1(m_1,\ldots,m_n)\le J_2(m_1,\ldots,m_i-1,\ldots,m_n)\]
and 
\[J_2(m_1,m_2,\ldots,m_n)\le J_1(m_1,\ldots,m_n)\le J_2(m_1,\ldots,m_j-1,\ldots,m_n)\]
\end{itemize}
\end{theorem}
\begin{proof}
We begin with part (a). If $\LL_1$ and $\LL_2$ differ by a single positive crossing change involving the component $L_{1i}$,
then there is an APSIC from $\LL_1$ to $\LL_2$. The construction is a generalization of \cite[Example 2.2]{BL2}. We take a product
cobordism between components $L_{1j}$ and $L_{2j}$ for $j\neq i$ and an annulus with a single positive double point connecting $L_{1i}$ to $L_{2i}$.
The concordance has $\newm_{kl}=0$ unless $k=l=i$, we have $\newm_{ii}=1$. In the notation of Theorem~\ref{thm:mainestimate}
we have $\rr=\ee_i$ and $\kk=(0,\ldots,0)$. Part (a) of Theorem~\ref{thm:single} follows immediately.

\smallskip
Part (b) is analogous. We construct an APSIC with $\newm_{kl}=0$ with the exception that $\newm_{ij}=\newm_{ji}=1$. 
We have $\rr=(0,\ldots,0)$ and the splitting $1=\newm_{ij}=\newm_{ij}^1+\newm_{ij}^2$ can be done in two ways: $(\newm_{ij}^1,\newm_{ij}^2)=(0,1)$ or $(1,0)$.
This gives two possibilities for choosing $\kk$, namely $\kk=\ee_i$ or $\kk=\ee_j$. Applying Theorem~\ref{thm:mainestimate} concludes the proof.
\end{proof}
\section{Splitting numbers of links}\label{sec:splitting}
Let us recall the following definition.
\begin{definition} Let $\LL$ be a link with $n$ components.
\begin{itemize}
\item The \emph{splitting number} $sp(\LL)$ is the minimal number of \emph{multicolored} crossing changes (that is, between different components)
needed to turn $\LL$ into a split link.
\item The \emph{clasp number} is the minimal number of double points of a singular concordance between $\LL$ and an unlink with the same
number of components.
\end{itemize}
\end{definition}
\begin{example}
The splitting number of the Whitehead link is $2$, even though the unlinking number is $1$. 
\end{example}

We will use the following terminology:
\begin{definition}
A \emph{positive crossing change} is a change of a negative crossing of a link into a positive crossing. Likewise, a \emph{negative crossing change}
is a change of a positive crossing into a negative crossing.
\end{definition}
\subsection{Splitting number bound from the $\wt{J}$-function}
In Definition~\ref{def:wtj}
we defined a $\wt{J}$-function of a link. The following result gives a ready-to-use bound for the splitting number.

\begin{theorem}\label{thm:basicsplit}
Suppose that $\LL$ can be turned into an unlink using $t_+$ positive and $t_-$ negative multicolored crossing changes.
Then $-t_-\le\wt{J}(\mm)\le t_+$ for all $\mm\in\Z^n$.
\end{theorem}
\begin{proof}
Use Theorem~\ref{thm:single} together with Proposition~\ref{prop:changeby1} (the latter holds for the \JFF function as well, because
$J$ differs from $H$ by an overall argument shift). We obtain that if two links $\LL_1$ and $\LL_2$ differ by a single positive multicolored crossing
change, then for all $\mm\in\Z^n$
\[J_2(\mm)\le J_1(\mm)\le J_2(\mm)+1.\]
Notice that a multicolored crossing change of a link $\LL$ does not affect the isotopy type of its components, in particular
the \JFF functions of its components remain the same. This shows that if $\LL_1$ and $\LL_2$ differ by a single positive multicolored crossing change,
then for all $\mm\in\Z^n$:
\[\wt{J}_2(\mm)\le \wt{J}_1(\mm)\le \wt{J}_2(\mm)+1.\]
Using this result repeatedly we show that if $\LL_1$ and $\LL_2$ differ by $t_+$ positive multicolored crossing changes and $t_-$ negative
multicolored crossing changes, then
\[\wt{J}_2(\mm)-t_-\le \wt{J}_1(\mm)\le \wt{J}_2(\mm)+t_+.\]
Suppose now that $\wt{J}_2$ is a split link. Then by Corollary~\ref{cor:splitwt} we know that $\wt{J}_2=0$. In particular
\[-t_-\le \wt{J}_1(\mm)\le t_+.\]
\end{proof}

Theorem~\ref{thm:basicsplit} is very useful for quick estimates of the splitting number of L-space links with two components, because then
the $\wt{J}$-function can be quickly determined from the Alexander polynomial. 
\begin{example}
We continue the example of Whitehead link, see Example~\ref{ex:whitehead}. 
As the Alexander polynomial is $\Delta=-(t_1-1)(t_2-1)t_1^{-1/2}t_2^{-1/2}$, by Corollary \ref{cor:two} we have $\wt{J}(t_1,t_2)=1$. 
This shows that the splitting number of the Whitehead link is at least $1$.
\end{example}


\subsection{Two-component links}

Theorem~\ref{thm:single} can be used directly to obtain some bounds for splitting numbers
for two-component links. 

\begin{theorem}
\label{th:2 comp}
Let $\LL$ be an arbitrary link with two components, and a link $\LL'$ can be obtained
from $\LL$ by changing $a$ negative \emph{multicolored} crossings to positive, and $b$ positive \emph{multicolored} crossings to negative.
Write $a=a_1+a_2$ and $b=b_1+b_2$ for arbitrary nonnegative $a_i,b_i$, then the following inequalities hold for all $m_1,m_2$:
$$
J'(m_1+b_1,m_2+b_2)\le J(m_1,m_2)\le J'(m_1-a_1,m_2-a_2).
$$
\end{theorem}

\begin{proof}
It is sufficient to consider a single crossing change. If $a_1=1,a_2=b_1=b_2=0$, then by Theorem \ref{thm:single} one has
$$
J'(m_1,m_2)\le J(m_1,m_2)\le J'(m_1-1,m_2).
$$ 
If $b_1=1,a_1=a_2=b_2=0$, then by Theorem \ref{thm:single} one has
$$
J(m_1,m_2)\le J'(m_1,m_2)\le J(m_1-1,m_2),
$$
so
$$
J'(m_1+1,m_2)\le J(m_1,m_2)\le J'(m_1,m_2).
$$
\end{proof}

The following corollary will be useful below:

\begin{theorem}
\label{th: vanishing}\
\begin{itemize}
\item[(a)] Suppose that a two component link  $\LL=L_1\cup L_2$ can be unlinked using $a$ positive and $b$ negative crossing changes.
Let $g_i$ denote the slice genus of $L_i$. Define vectors
$$
\ggg:=(g_1,g_2),\ \wt{\ggg}:=\left(g_1+\frac 12\lk(L_1,L_2),g_2+\frac 12\lk(L_1,L_2)\right).
$$
Define the region $R(a)$  by inequalities:
$$
R(a):=\{(m_1,m_2):m_1+m_2\ge a, m_1\ge 0, m_2\ge 0\};
$$
see Figure \ref{fig: J vanish}. Then $J(\mm)=\wt{J}(\mm)=0$ for $\mm\in R(a)+\ggg$.

\item[(b)] If, in addition, $\LL$ is an L-space link, then 
$$\HFL(\LL,\vv)=0\ \text{for}\  \vv \in R(a)+\wt{\ggg}+(1,1).$$
In particular, all coefficients of the Alexander polynomial vanish in $R(a)+\wt{\ggg}+(\frac 12,\frac 12)$.
\end{itemize}
\end{theorem}

\begin{proof}
As above, let $J_i$ denote the \JFF functions for the components  of $\LL$. 
For a split link $\LL'=L_1\sqcup L_2$ we get  
$
J'(v_1,v_2)=J_1(v_1)+J_2(v_2).
$
Furthermore, by \cite[Corollary 7.4]{Ras} we get $J_i(v_i)=0$ for $v_i\ge g_i$.

Assume that $\mm=(m_1,m_2)$ belongs to $R(a)+\ggg$. By definition, $J(m_1,m_2)\ge 0$. 
On the other hand, let us choose
$a_1=\min(m_1-g_1,a)$ and $a_2=a-a_1$, then $m_1-a_1\ge g_1$ and 
$$m_2-a_2=m_2-a+\min(m_1-g_1,a)\ge m_2-a+m_1-g_1\ge g_2.$$
Therefore by Theorem \ref{th:2 comp}:
$$J(m_1,m_2)\le J'(m_1-a_1,m_2-a_2)=J_1(m_1-a_1)+J_2(m_2-a_2)=0.$$
Since $J_1(m_1)=J_2(m_2)=0$, we get also get $\wt{J}(\mm)=0$. 

Suppose now that $\LL$ is an L-space link. By the above, \HFF vanishes in $\vv\in R(a)+\wt{\ggg}$.
Corollary~\ref{cor:two} implies the vanishing of the coefficients of the Alexander polynomial
in $R(a)+\wt{\ggg}+(\frac12,\frac12)$. To show that $\HFL(\vv)=0$ for $\vv\in R(a)+\wt{\ggg}+(1,1)$,
remark that for such $\vv$ one has $H(\vv-\ee_i)=H(\vv)=0$, so the natural inclusions $A^-(\vv-\ee_i)\hookrightarrow A^-(\vv)$
induce isomorphisms on homology. By \eqref{def of HFL}, $\HFL(\vv)=0$.
\end{proof}

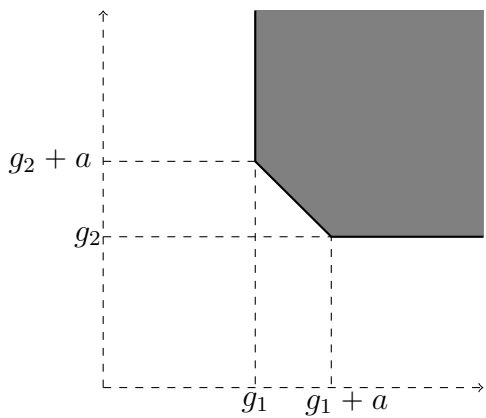
\begin{figure}
\begin{tikzpicture}
\draw [->,dashed](0,0)--(5,0);
\draw [->,dashed](0,0)--(0,5);
\draw [dashed] (0,2)--(3,2)--(3,0);
\draw [dashed] (0,3)--(2,3)--(2,0);
\filldraw [color=gray] (2,5)--(2,3)--(3,2)--(5,2)--(5,5)--(2,5);
\draw  [thick] (2,5)--(2,3)--(3,2)--(5,2);
\draw (2,-0.2) node {$g_1$};
\draw (3.2,-0.2) node {$g_1+a$};
\draw (-0.2,2) node {$g_2$};
\draw (-0.7,3) node {$g_2+a$};
\end{tikzpicture}
\caption{Region $R(a)+\ggg$ where $J$ and $\wt{J}$ vanish.}
\label{fig: J vanish}
\end{figure}

\begin{remark}
Part (b) of Theorem~\ref{th: vanishing} does not hold for non L-space links. For example, the link $L9a31$ in \cite{linkinfo}
has two components, one being an unknot and one being a trefoil. The linking number of the components is $1$, so $\wt{\ggg}=(\frac12,\frac32)$. 
Now the Alexander polynomial is
\[t_1^{-1}t_2^{-2}-t_2^{-2}-2t_1^{-1}t_2^{-1}+4t_2^{-1}-t_1t_2^{-1}+2t_1^{-1}-5+2t_1-t_1^{-1}t_2+4t_2-2t_1t_2-t_2^2+t_1t_2^2.\]
The term $t_1t_2^2$ has exponents $(1,2)$ which belong to $R(0)+\wt{\ggg}+(\frac12,\frac12)$. Therefore, Theorem~\ref{th: vanishing} would
imply that we need at least one crossing change from negative to positive in order to split $L9a31$. 
Such a crossing change will make the two components have linking number $2$, so we will need at least two more crossing changes
to make the linking number equal to $0$. Altogether,
we would need at least three crossing changes to split $L9a31$.
However, we can split $L9a31$ is a single move.
\end{remark}
\subsection{Example: the two-bridge link $b(8,-5)$}\label{sec:b85}

We will discuss an example of the two-bridge link $b(8,-5)$ which was shown by Liu \cite[Example 3.8]{Liu} to be an L-space link. It
is presented in Figure~\ref{fig:b85}. The orientation of $b(8,-5)$ is as in \cite{Liu}. The two components have linking number $0$.
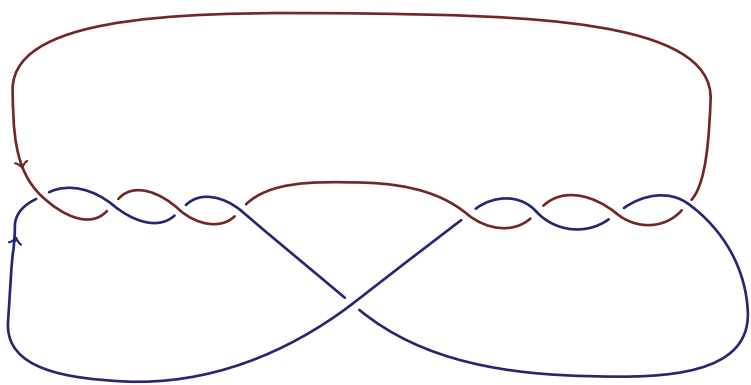
\begin{figure}
\definecolor{linkcolor0}{rgb}{0.45, 0.15, 0.15}
\definecolor{linkcolor1}{rgb}{0.15, 0.15, 0.45}
\begin{tikzpicture}[line width=1, line cap=round, line join=round]
  \begin{scope}[color=linkcolor0]
    \draw (9.11, 2.54) .. controls (9.31, 2.77) and (9.34, 3.40) .. 
          (9.36, 3.86) .. controls (9.41, 4.95) and (6.83, 4.98) .. 
          (4.79, 5.01) .. controls (2.76, 5.03) and (0.18, 5.06) .. 
          (0.18, 4.00) .. controls (0.19, 3.49) and (0.19, 2.94) .. (0.56, 2.59);
    \draw (0.56, 2.59) .. controls (0.83, 2.33) and (1.20, 2.15) .. (1.42, 2.39);
    \draw (1.57, 2.55) .. controls (1.77, 2.78) and (2.13, 2.64) .. (2.38, 2.40);
    \draw (2.38, 2.40) .. controls (2.60, 2.21) and (2.91, 2.14) .. (3.10, 2.32);
    \draw (3.25, 2.48) .. controls (3.56, 2.78) and (4.13, 2.78) .. 
          (4.62, 2.77) .. controls (5.17, 2.77) and (5.75, 2.70) .. (6.17, 2.34);
    \draw (6.17, 2.34) .. controls (6.42, 2.14) and (6.76, 2.09) .. (6.99, 2.29);
    \draw (7.16, 2.46) .. controls (7.42, 2.71) and (7.83, 2.60) .. (8.12, 2.36);
    \draw (8.12, 2.36) .. controls (8.38, 2.14) and (8.76, 2.15) .. (8.98, 2.40);
    \draw[->] (0.3, 3.00) -- +(0.01, -0.05);
  \end{scope}
  \begin{scope}[color=linkcolor1]
    \draw (4.65, 1.16) .. controls (3.78, 0.49) and (2.74, 0.08) .. 
          (1.65, 0.13) .. controls (0.91, 0.17) and (0.08, 0.28) .. 
          (0.12, 0.92) .. controls (0.14, 1.21) and (0.16, 1.62) .. 
          (0.18, 1.80) .. controls (0.20, 1.93) and (0.21, 2.06) .. 
          (0.21, 2.19) .. controls (0.21, 2.36) and (0.34, 2.48) .. (0.49, 2.55);
    \draw (0.66, 2.64) .. controls (0.93, 2.77) and (1.26, 2.67) .. (1.50, 2.47);
    \draw (1.50, 2.47) .. controls (1.75, 2.26) and (2.09, 2.13) .. (2.31, 2.33);
    \draw (2.46, 2.47) .. controls (2.65, 2.66) and (2.96, 2.58) .. (3.18, 2.40);
    \draw (3.18, 2.40) .. controls (3.63, 2.01) and (4.09, 1.63) .. (4.55, 1.24);
    \draw (4.74, 1.08) .. controls (5.58, 0.38) and (6.71, 0.22) .. 
          (7.80, 0.20) .. controls (8.77, 0.17) and (9.87, 0.22) .. 
          (9.85, 1.04) .. controls (9.83, 1.62) and (9.53, 2.15) .. (9.06, 2.49);
    \draw (9.06, 2.49) .. controls (8.81, 2.67) and (8.47, 2.61) .. (8.22, 2.43);
    \draw (8.02, 2.28) .. controls (7.73, 2.07) and (7.32, 2.11) .. (7.07, 2.38);
    \draw (7.07, 2.38) .. controls (6.87, 2.61) and (6.52, 2.61) .. (6.27, 2.42);
    \draw (6.08, 2.27) .. controls (5.60, 1.90) and (5.12, 1.53) .. (4.65, 1.16);
    \draw[->] (0.22, 2.00) -- +(0.01, 0.05);
  \end{scope}
\end{tikzpicture}
\caption{The link $b(8,-5)$. Its two components are unknots.}\label{fig:b85}
\end{figure}
In the notation of LinkInfo \cite{linkinfo} it is the link $L9a40$. It was shown in \cite[Section 7.1]{CFP} that
the splitting number of this link is $4$. The tool was studying the smooth four genus of the link obtained by taking a double branch
cover of one of the components of $b(8,-5)$. The splitting number of $b(8,-5)$
can be also detected by the signatures as in \cite{CCZ}. We will show that $sp(b(8,-5))=4$ using the \JFF function.

The Alexander polynomial of $b(8,-5)$ can be found on the LinkInfo web page \cite{linkinfo} or 
calculated using the SnapPy package \cite{snappy}. We have
\[\Delta(t_1,t_2)=-(t_1+t_2+1+t_1^{-1}+t_2^{-1})(t_1^{1/2}-t_1^{-1/2})(t_2^{1/2}-t_2^{-1/2}).\]
After normalizing, by Corollary~\ref{cor:two} the generating function for the $\wt{J}$-function equals
\begin{equation}\label{eq:wtjofb}
\wt{\JJ}(t_1,t_2)=t_1+t_2+1+t_1^{-1}+t_2^{-1}.
\end{equation}
Theorem~\ref{thm:basicsplit} implies that we need to make at least one positive crossing change to unlink $b(8,-5)$. As the original linking number
is zero and a positive crossing change increases the linking number, we have to compensate the positive crossing change with a negative crossing change,
so the splitting number is at least $2$. That is all we can deduce from Theorem~\ref{thm:basicsplit}.

On the other hand, $J(1,0)=1$, so by Theorem \ref{th: vanishing} one needs at least two positive crossing changes to split $b(8,-5)$. As each
such crossing change increases the linking number between the two components of $b(8,-5)$, we also need two negative crossing changes. Therefore
we have proved the following result.
\begin{proposition}
The splitting number of $b(8,-5)$ is at least $4$.
\end{proposition}
It is quite easy to split the $b(8,-5)$ in four moves.
\begin{remark}
SnapPy and and the LinkInfo webpage
\cite{linkinfo} give the Alexander polynomial of $b(8,-5)$ with opposite sign. To choose the sign we notice that the other choice
of sign of the Alexander polynomial yields $\wt{\JJ}$ with negative coefficients only, hence, for example $J(0,0)=-1$. This
contradicts property of non-negativity of the \JFF function. Liu's algorithm in \cite[Section 3.3]{Liu0}
gives the proper sign of the Alexander polynomial.
\end{remark}


\begin{figure}
\definecolor{linkcolor0}{rgb}{0.45, 0.15, 0.15}
\definecolor{linkcolor1}{rgb}{0.15, 0.15, 0.45}
\begin{tikzpicture}[line width=1, line cap=round, line join=round]
  \begin{scope}[color=linkcolor0]
    \draw (9.11, 2.54) .. controls (9.31, 2.77) and (9.34, 3.40) .. 
          (9.36, 3.86) .. controls (9.41, 4.95) and (6.83, 4.98) .. 
          (4.79, 5.01) .. controls (2.76, 5.03) and (0.18, 5.06) .. 
          (0.18, 4.00) .. controls (0.19, 3.49) and (0.19, 2.94) .. (0.56, 2.59);
    \draw (0.56, 2.59) .. controls (0.83, 2.33) and (1.20, 2.15) .. (1.42, 2.39);
    \draw (1.57, 2.55) .. controls (1.77, 2.78) and (2.13, 2.64) .. (2.38, 2.40);
    \draw (2.38, 2.40) .. controls (2.60, 2.21) and (2.91, 2.14) .. (3.10, 2.32);
    \draw (3.25, 2.48) .. controls (3.56, 2.78) and (4.13, 2.78) .. 
          (4.62, 2.77) .. controls (5.17, 2.77) and (5.75, 2.70) .. (6.17, 2.34);
    \draw (6.17, 2.34) .. controls (6.42, 2.14) and (6.76, 2.09) .. (6.99, 2.29);
    \draw (7.16, 2.46) .. controls (7.42, 2.71) and (7.83, 2.60) .. (8.12, 2.36);
    \draw (8.12, 2.36) .. controls (8.38, 2.14) and (8.76, 2.15) .. (8.98, 2.40);
    \draw[->] (0.3, 3.00) -- +(0.01, -0.05);
  \end{scope}
  \begin{scope}[color=linkcolor1]
    \draw (4.65, 1.16) .. controls (3.78, 0.49) and (2.74, 0.08) .. 
          (1.65, 0.13) .. controls (0.91, 0.17) and (0.08, 0.28) .. 
          (0.12, 0.92) .. controls (0.14, 1.21) and (0.16, 1.62) .. 
          (0.18, 1.80) .. controls (0.20, 1.93) and (0.21, 2.06) .. 
          (0.21, 2.19) .. controls (0.21, 2.36) and (0.34, 2.48) .. (0.49, 2.55);
    \draw (0.66, 2.64) .. controls (0.93, 2.77) and (1.26, 2.67) .. (1.50, 2.47);
    \draw (1.50, 2.47) .. controls (1.75, 2.26) and (2.09, 2.13) .. (2.31, 2.33);
    \draw (2.46, 2.47) .. controls (2.65, 2.66) and (2.96, 2.58) .. (3.18, 2.40);
    \draw (3.18, 2.40) .. controls (3.63, 2.01) and (4.09, 1.63) .. (4.55, 1.24);
    \draw (4.74, 1.08) .. controls (5.58, 0.38) and (6.71, 0.22) .. 
          (7.80, 0.20) .. controls (8.77, 0.17) and (9.87, 0.22) .. 
          (9.85, 1.04) .. controls (9.83, 1.62) and (9.53, 2.15) .. (9.06, 2.49);
    \draw (9.06, 2.49) .. controls (8.81, 2.67) and (8.47, 2.61) .. (8.22, 2.43);
    \draw (8.02, 2.28) .. controls (7.73, 2.07) and (7.32, 2.11) .. (7.07, 2.38);
    \draw (7.07, 2.38) .. controls (6.87, 2.61) and (6.52, 2.61) .. (6.27, 2.42);
    \draw (6.08, 2.27) .. controls (5.60, 1.90) and (5.12, 1.53) .. (4.65, 1.16);
    \draw[->] (0.22, 2.00) -- +(0.01, 0.05);
  \end{scope}
    \draw[fill=white] (1,2.2) -- (1,2.8) -- (2,2.8) node[scale=0.8,below=0.1] {$n$ full twists} -- (3,2.8) -- (3,2.2) -- cycle;
    \draw[fill=white] (6.7,2.1) -- (6.7,2.8) -- (7.7,2.8) node[scale=0.8,below=0.1] {$k$ full twists} -- (8.7,2.8) -- (8.7,2.1) -- cycle;
\end{tikzpicture}
\caption{The general $b(rq-1,-q)$ two-bridge link, where $r=2n+1$, $q=2k+1$. The linking number is $n-k$.}\label{fig:bnk}
\end{figure}
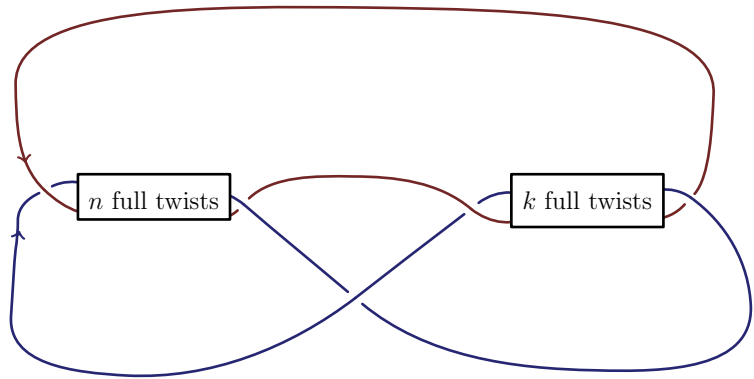

\begin{figure}[t]
\begin{tikzpicture}[scale=.5]
  \draw[->, color=black!20] (0,0) -- (14.5,0) node[right] {$t$};
  \draw[->, color=black!20] (0,0) -- (0,15.5) node[above] {$q$};
  \draw[step=1, color=black!20] (0,0) grid (14,15);
  \draw (0.5,-.2) node[below] {$-6$};
  \draw (1.5,-.2) node[below] {$-5$};
  \draw (2.5,-.2) node[below] {$-4$};
  \draw (3.5,-.2) node[below] {$-3$};
  \draw (4.5,-.2) node[below] {$-2$};
  \draw (5.5,-.2) node[below] {$-1$};
  \draw (6.5,-.2) node[below] {$0$};
  \draw (7.5,-.2) node[below] {$1$};
  \draw (8.5,-.2) node[below] {$2$};
  \draw (9.5,-.2) node[below] {$3$};
  \draw (10.5,-.2) node[below] {$4$};
  \draw (11.5,-.2) node[below] {$5$};
  \draw (12.5,-.2) node[below] {$6$};
  \draw (13.5,-.2) node[below] {$7$};
  \draw (-.2,0.5) node[left] {$-8$};
  \draw (-.2,1.5) node[left] {$-6$};
  \draw (-.2,2.5) node[left] {$-4$};
  \draw (-.2,3.5) node[left] {$-2$};
  \draw (-.2,4.5) node[left] {$0$};
  \draw (-.2,5.5) node[left] {$2$};
  \draw (-.2,6.5) node[left] {$4$};
  \draw (-.2,7.5) node[left] {$6$};
  \draw (-.2,8.5) node[left] {$8$};
  \draw (-.2,9.5) node[left] {$10$};
  \draw (-.2,10.5) node[left] {$12$};
  \draw (-.2,11.5) node[left] {$14$};
  \draw (-.2,12.5) node[left] {$16$};
  \draw (-.2,13.5) node[left] {$18$};
  \draw (-.2,14.5) node[left] {$20$};
  \fill (0.5, 0.5) circle (.15);
  \fill (0.5, 1.5) circle (.15);
  \fill (1.5, 1.5) circle (.15);
  \fill (1.5, 2.5) circle (.15);
  \draw (2.5, 2.5) node {$2$};
  \draw (2.5, 3.5) node {$2$};
  \draw (3.5, 3.5) node {$3$};
  \draw (3.5, 4.5) node {$3$};
  \draw (4.5, 4.5) node {$4$};
  \draw (4.5, 5.5) node {$4$};
  \draw (5.5, 5.5) node {$5$};
  \draw (5.5, 6.5) node {$5$};
  \draw (6.5, 6.5) node {$6$};
  \draw (6.5, 7.5) node {$6$};
  \draw (7.5, 7.5) node {$5$};
  \draw (7.5, 8.5) node {$5$};
  \draw (8.5, 8.5) node {$6$};
  \draw (8.5, 9.5) node {$6$};
  \draw (9.5, 9.5) node {$5$};
  \draw (9.5, 10.5) node {$5$};
  \draw (10.5, 10.5) node {$4$};
  \draw (10.5, 11.5) node {$4$};
  \draw (11.5, 11.5) node {$3$};
  \draw (11.5, 12.5) node {$3$};
  \draw (12.5, 12.5) node {$2$};
  \draw (12.5, 13.5) node {$2$};
  \fill (13.5, 13.5) circle (.15);
  \fill (13.5, 14.5) circle (.15);
  \draw[->] (2.288, 2.288) -- (1.712, 1.712);
  \draw[->] (2.288, 3.288) -- (1.712, 2.712);
  \draw[->] (4.288, 4.288) -- node[color=red!75!black] {$3$} (3.712, 3.712);
  \draw[->] (4.288, 5.288) -- node[color=red!75!black] {$3$} (3.712, 4.712);
  \draw[->] (5.288, 5.288) -- (4.712, 4.712);
  \draw[->] (5.288, 6.288) -- (4.712, 5.712);
  \draw[->] (6.288, 6.288) -- node[color=red!75!black] {$4$} (5.712, 5.712);
  \draw[->] (6.288, 7.288) -- node[color=red!75!black] {$4$} (5.712, 6.712);
  \draw[->] (7.288, 7.288) -- node[color=red!75!black] {$2$} (6.712, 6.712);
  \draw[->] (7.288, 8.288) -- node[color=red!75!black] {$2$} (6.712, 7.712);
  \draw[->] (8.288, 8.288) -- node[color=red!75!black] {$3$} (7.712, 7.712);
  \draw[->] (8.288, 9.288) -- node[color=red!75!black] {$3$} (7.712, 8.712);
  \draw[->] (9.288, 9.288) -- node[color=red!75!black] {$3$} (8.712, 8.712);
  \draw[->] (9.288, 10.288) -- node[color=red!75!black] {$3$} (8.712, 9.712);
  \draw[->] (10.288, 10.288) -- node[color=red!75!black] {$2$} (9.712, 9.712);
  \draw[->] (10.288, 11.288) -- node[color=red!75!black] {$2$} (9.712, 10.712);
  \draw[->] (11.288, 11.288) -- node[color=red!75!black] {$2$} (10.712, 10.712);
  \draw[->] (11.288, 12.288) -- node[color=red!75!black] {$2$} (10.712, 11.712);
  \draw[->] (12.288, 12.288) -- (11.712, 11.712);
  \draw[->] (12.288, 13.288) -- (11.712, 12.712);
  \draw[->] (13.288, 13.288) -- (12.712, 12.712);
  \draw[->] (13.288, 14.288) -- (12.712, 13.712);
\end{tikzpicture}
\caption{The $E_1$ page of the Batson--Seed spectral sequence for the link $b(48,-7)$ with splitting number $3$. It can be shown
that $E_2=E_\infty$ is the Khovanov homology of the unknot. The calculations were made
using the KnotKit program \cite{knotkit}.}\label{fig:batson-seed}.
\end{figure}

\subsection{More general two-bridge links}

The arguments used in Section~\ref{sec:b85} can be easily generalized for the case of two-bridge links $\LL_n=b(4n^2+4n,-2n-1)$.
The components of $\LL_n$ are unknots with linking number 0. For example, $\LL_1$ is the Whitehead link.
It is proved in \cite[Section 3]{Liu} that all the $\LL_n$ are L-space links, and their Alexander polynomials were computed in \cite[Section 6]{Liu},
see also \cite[Section 3.3]{Liu0}:
$$
\Delta_{\LL_n}(t_1,t_2)=(-1)^{n}\sum_{|i+1/2|+|j+1/2|\le n}(-1)^{i+j}t_1^{i+1/2}t_2^{j+1/2}.
$$ 
Clearly,
$$
\Delta(t_1,t_2)=-t_1^{n-\frac12}\left(t_2^{1/2}-t_2^{-1/2}\right)+\text{terms of lower degree in}\ t_1,
$$
Hence by  Corollary \ref{cor:two}:
$$
\wt{\JJ}_{\LL_n}(t_1,t_2)=-\frac{\Delta_{\LL_n}(t_1,t_2)}{\left(t_1^{1/2}-t_1^{-1/2}\right)\left(t_2^{1/2}-t_2^{-1/2}\right)}=t_1^{n-1}+\text{terms of lower degree in}\ t_1,
$$
and $J(n-1,0)=\wt{J}(n-1,0)=1$. 
By Theorem \ref{th: vanishing} one needs at least $n$ positive crossing changes to split $\LL_n$,
and the linking number argument from the previous section implies that one needs same number of negative crossing changes. 
 We obtained the following result.
\begin{theorem}
\label{th:two bridge}
The splitting number of $\LL_n$ is at least $2n$.
\end{theorem}
It is quite easy to split the $\LL_n$ in $2n$ moves using Figure \ref{fig:bnk} (where $k=n$).
On the other hand, as all these links are alternating, the Batson--Seed spectral sequence degenerates at most at the $E_3$ page by 
Proposition~\ref{prop:degenerates}, so Theorem~\ref{thm:batsonseed} will not detect the splitting number of $\LL_n$ for $n>1$.

\subsection{Comparison with the Batson--Seed criterion}

In \cite{BS} Batson and Seed proved an obstruction for splitting number. For the sake of simplicity we formulate the result
for a link with two components and for homologies over $\Z_2$.
\begin{theorem}\cite{BS}
\label{thm:batsonseed}
Suppose $\LL=L_1\cup L_2$ is a two component link and let $\LL'$ be a split link with components $L_1,L_2$. Then
there exists a spectral sequence, whose $E_1$ page is the Khovanov homology $Kh(L)$ and $E_\infty$ page is the Khovanov
homology $Kh(\LL')$. If the splitting number of $\LL$ is $k$, then the $E_k$ page is equal to the $E_\infty$ page of this sequence.
\end{theorem}

In Figure~\ref{fig:batson-seed} there is shown the $E_1$ page of the Batson--Seed spectral sequence for $b(48,-7)$, whose
splitting number was shown to be $6$. The arrows in the figure correspond to the differential. 
We have $E_2=E_\infty$, so Theorem~\ref{thm:batsonseed} implies that $sp(b(8,-5))\ge 2$. This means that the Batson--Seed
criterion does not detect the splitting number of $b(48,-7)$.

For general two-bridge links $b(4n^2+4n,-2n-1)$ 
we have the following observation, which limits the usage of the Batson--Seed criterion. It is well known to the experts.

\begin{proposition}\label{prop:degenerates}
Suppose $\LL$ is an alternating non-split link. Then the Batson--Seed spectral sequence collapses at most at the $E_3$ page.
\end{proposition}
\begin{proof}
By \cite{Lee} $\LL$ is Khovanov thin, that is, the Khovanov homology is supported on two diagonals. More precisely, if $x$ is
a non-trivial element of $Kh(\LL)$, then $q(x)=2h(x)-\sigma(\LL)\pm 1$, where $q(x)$ is the $q$-grading, $h(x)$ is the homological grading
and $\sigma(\LL)$ is the signature of $\LL$.

The differential in the Batson--Seed spectral sequence is $d=d_0+d_1$, where $d_0$ is the standard differential in the Khovanov complex and $d_1$
decreases the homological grading by $1$ and drops the $q$-grading by $2$. A higher differential $d_k$ changes the $(h,q)$ bigrading by $(1-2k,-2k)$, and hence changes the difference $q-2h$ by $2k-2$. As $\LL$ is thin, the only potentially non-trivial differentials are $d_0,d_1$ and $d_2$.
\end{proof}

\subsection{Cables on the Whitehead link}
As a more complicated example, we calculate the splitting number of cables on the Whitehead link.
Let $Wh_{p,q}$ denote the link consisting of the $(p,q)$ cable on one of the component of the Whitehead link
and the unchanged second component of it. It is clear that the linking number of $Wh_{p,q}$ vanishes.
A $(1,1)$-surgery on the Whitehead
link is an L-space \cite[Example 3.1]{Liu}, 
hence by Proposition~\ref{prop:cablesonlinks} $Wh_{p,q}$ is an L-space link as long as $1<p<q$ and $p,q$ are coprime.

The Whitehead link has Alexander polynomial $\Delta=-(t_1^{1/2}-t_1^{-1/2})(t_2^{1/2}-t_2^{-1/2})$. 
The Alexander polynomial of a cable link was calculated by Turaev in \cite[Theorem 1.3.1]{Tu86}.  
\begin{theorem}\label{thm:turaevcable}
Let $\LL=L_1\cup\ldots\cup L_n$ be a link and $\Delta_\LL(t_1,\ldots,t_n)$ its multivariable Alexander polynomial. Let $\LL_{p,q}$ be as in the statement of 
Proposition~\ref{prop:cablesonlinks} above. Set $T=t_1^qt_2^{l_2}\ldots t_n^{l_n}$, where $l_j=\lk(L_1,L_j)$. Then
\[\Delta_{\LL_{p,q}}(t_1,\ldots,t_n)=\Delta_\LL(t_1^p,t_2,t_3,\ldots,t_n)\frac{T^{p/2}-T^{-p/2}}{T^{1/2}-T^{-1/2}}.\]
\end{theorem}

It follows from the theorem that the Alexander polynomial of $Wh_{p,q}$ is equal to
\[\Delta_{Wh_{p,q}}=-\left(t_1^{p/2}-t_1^{-p/2}\right)\left(t_2^{1/2}-t_2^{-1/2}\right)\frac{t_1^{pq/2}-t_1^{-pq/2}}{t_1^{q/2}-t_1^{-q/2}}.\]
From this we obtain by Corollary~\ref{cor:two}.
\[\wt{\JJ}(t_1,t_2)=\frac{\left(t_1^{p/2}-t_1^{-p/2}\right)\left(t_1^{pq/2}-t_1^{-pq/2}\right)}{\left(t_1^{1/2}-t_1^{-1/2}\right)\left(t_1^{q/2}-t_1^{-q/2}\right)}=
t_1^{\delta+(p-1)}+\text{terms of lower degree in}\ t_1,\]
where $\delta=\frac12(p-1)(q-1)$. In particular, $\wt{J}(\delta+(p-1),0)=1$.

Now the genera of the components of $Wh_{p,q}$ are $g_1=\delta, g_2=0$. By Theorem \ref{th: vanishing}
we infer that we need to perform at least $p$ positive-to-negative multicolored crossing changes to transform $Wh_{p,q}$ into the disjoint sum
of $T(p,q)$ and the unknot. The linking number argument shows that we also need $p$ negative crossing changes, hence we obtain the
following result.
\begin{proposition}
\label{th: whitehead cable}
The splitting number of the $(p,q)$-cable on the Whitehead link is at least $2p$.
\end{proposition}
It is not hard to find a splitting sequence of length $2p$.

\section{Algebraic links}\label{sec:algebraiclinks}

\subsection{The \HFF function for links of singularities}
Let $C$ be a germ of a complex plane curve singularity with branches $C_1,\ldots,C_n$. 
Its intersection with a small sphere is called an algebraic link.
It is shown in 
\cite{GN2} that all algebraic links are L-space links. For algebraic links the \HFF 
function admits the following description. Let $\gamma_i:(\BC,0)\to (C_i,0)$ denote the uniformization of $C_i$. Define the set
$$
\J(\vv):=\{f\in \BC[[x,y]]: \Ord_{0}f(\gamma_i(t))\ge v_i\}
$$
It is clear that $\J(\vv)$ is in fact a vector subspace of $\BC[[x,y]]$.
Define
the {\em Hilbert function} of $C$ as 
\begin{equation}\label{eq:Rfunction}
R(\vv)=\codim \J(\vv)=\dim_{\C} \C[x,y]/\J(\vv).
\end{equation}
Moreover, set 
\[
R_i(v_i)=R(0,\ldots,0,v_i,0,\ldots,0).
\]
Notice that for a singularity with one branch, $R(\vv)=R_1(v_1)$ 
is the number of the elements of the semigroup of the singular point in the interval $[0,v_1)$,
so \eqref{eq:Rfunction} can be regarded as a generalization the definition of $R$--function in \cite{BL}.

We can relate $R$ to the \HFF function discussed above. Define
\begin{equation}
\label{eq:ggg}
\ggg=(g_1,\ldots,g_n);\quad \wt{\ggg}=(\wt{g}_1,\ldots,\wt{g}_n),\ \wt{g}_i=g_i+\frac12\sum_{j\neq i}\lk(L_i,L_j).
\end{equation}
where $g_i$ is the Seifert genus of $L_i$ (or, equivalently, the delta-invariant of the singularity $C_i$). It is known
that for algebraic links $2\wt{\ggg}$ is the conductor of the multi-dimensional semigroup of~$C$; see
for instance \cite[Chapter 17]{Kunz}. Campillo, Delgado and Gusein-Zade related \cite{CDG} the Alexander polynomial
of an algebraic link to the semigroup of the corresponding curve. Based on their result and \eqref{invertion}, the following 
formula for the Hilbert function was obtained in \cite{GN}:

\begin{theorem}[see \cite{GN}]
For an algebraic link, one has
\begin{equation}\label{eq:RandH}
H(\vv)=R(\wt{\ggg}-\vv),\ J(\vv)=R(\ggg-\vv).
\end{equation}
\end{theorem}

\begin{remark}
It was proven in \cite{CDK} that for all plane curve singularities the Hilbert function satisfies the following symmetry property:
\begin{equation}
\label{eq: symmetry for R}
R(2\wt{\ggg}-\vv)=R(\vv)+|\wt{\ggg}|-|\vv|.
\end{equation}
Indeed, this agrees with the symmetry property \eqref{duality} of $H$.
\end{remark}

\begin{theorem}
\label{th:hilb}
We have the following inequality for the Hilbert function of a plane curve singularity.
$$
0\ge R(\vv)-\sum_{i=1}^{n}R_i(v_i)\ge -\sum_{i<j}\lk(L_i,L_j).
$$
Both inequalities are sharp.
\end{theorem}

\begin{corollary}\label{cor:algebraicJ}
For an algebraic link, for all $\vv$:
$$
0\ge \wt{J}(\vv) \ge -\sum_{i<j}\lk(L_i,L_j)
$$
\end{corollary}

\begin{proof}
By \eqref{eq:RandH},
$J(\vv) =R(\ggg-\vv).$
Similarly, $J_i(\vv_i)=R_i(g_i-v_i)$, so it remains to apply the theorem to the vector $\ggg-\vv$.
\end{proof}

\begin{remark}
Corollary~\ref{cor:algebraicJ} can be compared with Theorem \ref{thm:basicsplit}. Indeed, all crossings in an algebraic link are positive, so
$t_+=0$, and by the above corollary we get $t_-\ge \sum_{i<j}\lk(L_i,L_j)$. In other words, to split an algebraic link
one needs to change exactly $\sum_{i<j}\lk(L_i,L_j)$ crossings from positive to negative. It is well known
that the splitting number of an algebraic link is equal to $\sum_{i<j}\lk(L_i,L_j)$.
\end{remark}

The following two lemmas will be used in the proof of Theorem~\ref{th:hilb}. 

\begin{lemma}
\label{l:hilb}
For $\uu,\vv\in \BZ^n$, one has
$$
R(\uu)+R(\vv)\ge R(\min(\uu,\vv))+R(\max(\uu,\vv)).
$$
\end{lemma}

\begin{proof}
Indeed, $\J(\uu),\J(\vv)\subset \J(\min(\uu,\vv))$ and $\J(\uu)\cap \J(\vv)=\J(\max(\uu,\vv))$. One has
$$
\dim \J(\min(\uu,\vv))/\J(\uu)+\dim \J(\min(\uu,\vv))/\J(\vv)\ge \dim \J(\min(\uu,\vv))/\J(\max(\uu,\vv)),
$$
therefore
$$
-R(\min(\uu,\vv))+R(\uu)-R(\min(\uu,\vv))+R(\vv)\ge -R(\min(\uu,\vv))+R(\max(\uu,\vv)).
$$
\end{proof}

\begin{lemma}
\label{l:cube}
Suppose that $\uu,\vv\in \BZ^n, 0\preceq \uu\preceq \vv$. 
Then 
\[R(\vv)-R(\uu)\le \sum_{i=1}^{n}(R_i(v_i)-R_i(u_i)).\]
\end{lemma}

\begin{proof}
Consider a sequence of lattice points $\vv^{(i)}=(u_1,\ldots,u_i,v_{i+1},\ldots, v_n)$, so that 
$\uu=\vv^{(n)}$ and $\vv=\vv^{(0)}$. Let $\ee_i$ denote the $i$-th coordinate vector.
Then 
\[
\max(\vv^{(i)}, v_i\ee_i)=\vv^{(i-1)},\ \min(\vv^{(i)},v_i\ee_i)=u_i\ee_i,
\]
hence by Lemma \ref{l:hilb}:
\begin{align*}
R_i(v_i)+R(\vv^{(i)})&\ge R(\vv^{(i-1)})+R_i(u_i),\\
\intertext{so}
R(\vv^{(i-1)})-R(\vv^{(i)})&\le R_i(v_i)-R_i(u_i),
\end{align*}
and
\[
R(\vv)-R(\uu)=R(\vv^{(0)})-R(\vv^{(r)})=\sum_{i=1}^{n}(R(\vv^{(i-1)})-R(\vv^{(i)}))\le \sum_{i=1}^{n} (R_i(v_i)-R_i(u_i)).
\]
\end{proof}

\begin{proof}[Proof of Theorem \ref{th:hilb}]
By Lemma \ref{l:cube}, one has
\[
R(\vv)=R(\vv)-R(0)\le \sum_{i=1}^{n}(R_i(v_i)-R_i(0))=\sum_{i=1}^{n}R_i(v_i).
\]
Furthermore, if $\uu\ggcurly\mathbf{0}$,
then by \eqref{eq: symmetry for R} $R(\uu)=|\uu|-|\wt{\ggg}|$.
By Lemma \ref{l:cube}, we get
\begin{align*}
R(\uu)-R(\vv)&\le \sum_{i=1}^{n}(R_i(u_i)-R_i(v_i)),\textrm{ so}\\
|\uu|-|\wt{\ggg}|-R(\vv)&\le \sum_{i=1}^{n}(u_i-g_i-R_i(v_i)),\textrm{ that is}\\
R(\vv)-\sum_{i=1}^{n}R_i(v_i)&\ge 
-|\wt{\ggg}|+\sum_{i=1}^{n}g_i=-\sum_{i<j}l_{ij}.
\end{align*}
\end{proof}

\subsection{Semicontinuity of the Hilbert function}
We can use Theorem~\ref{thm:mainestimate} to give a topological proof of semigroup semicontinuity property under $\delta$-constant
deformation, generalizing the result of \cite{BL2} for many components. We refer the reader to \cite{BCG,GN} for other approaches
to semicontinuity property of semigroups.
 
Suppose $F_t\colon(\C^2,0)\to (\C,0)$ is a family of polynomials depending on a parameter $t\in(-\varepsilon,\varepsilon)\subset\R$.
We assume that for every $t$ the curve $F_t^{-1}(0)$ has an isolated singularity with $n$ branches.
\begin{theorem}\label{thm:semigroupsemic}
Assume that the deformation
is $\delta$-constant. Then for any $\mm\in\Z^n$ and $t\neq 0$ sufficiently close to $0$ we have
\[R_t(\mm)\ge R_0(\mm).\]
\end{theorem}
\begin{proof}
We follow the proof of \cite[Theorem 2.15]{BL}. Take a ball $B\subset\C^2$ with center at $0$
such that $F_0^{-1}(0)\cap\partial B$ is the link
of the singularity of $F_0^{-1}(0)$ at $0$. Denote this link by $\LL_2$.

Choose $t$ sufficiently small. Then $F_t^{-1}(0)\cap\partial B$ is still isotopic to $\LL_2$. Choose a smaller ball $B'$ with center at $0$
such that $F_t^{-1}(0)\cap\partial B'$ is the link of the singularity of $F^{-1}_t(0)$ at $0$. Denote this link by $\LL_1$. For $i=1,2$,
let $L_{i1},\ldots,L_{in}$ be the components of the link $\LL_i$. Denote by $g_{i1},\ldots,g_{in}$ the Seifert genus of the corresponding component.
Let $\ggg$ be as in \eqref{eq:ggg}.
\begin{lemma}\label{lem:apsicfromdeformation}
Up to perturbing $F_t^{-1}(0)$ by a polynomial, the curve $F_t^{-1}(0)\cap (B\setminus B')$ is an APSIC from $\LL_1$ to $\LL_2$. The
number of monochromatic double points of the $i$-th component (denoted by $\newm_{ii}$ in Section~\ref{sec:psic}) is equal to $g_{i2}-g_{i1}$.
\end{lemma}
The proof of Lemma~\ref{lem:apsicfromdeformation} is given after the proof of Theorem~\ref{thm:semigroupsemic}.
Given the lemma we use Theorem~\ref{thm:mainestimate} to obtain
\begin{equation}\label{eq:J21}
J_2(\mm')\le J_1(\mm'-\kk),
\end{equation}
where $\kk=(\newm_{11},\ldots,\newm_{nn})=\ggg_2-\ggg_1$ and $\mm'\in\Z^n$ is arbitrary.
Substituting \eqref{eq:RandH} into \eqref{eq:J21}
we obtain
\[R_2(\mm)=J_2(\ggg_2-\mm)\le J_1(\ggg_2-\mm-\kk)=J_1(\ggg_1-\mm)=R_1(\mm).\]
\end{proof}
\begin{proof}[Proof of Lemma~\ref{lem:apsicfromdeformation}]
The proof is a direct generalization of \cite[Lemma 2.3]{BL2}. As the deformation is $\delta$-constant, we can find
a complex parametrization $\psi$ of $F_t^{-1}(0)\cap B''$ (where $B''$ is a ball slightly larger than $B$), 
whose domain is a disjoint union of $n$ disks $D_1,\ldots,D_n$. Set $D=D_1\sqcup\ldots\sqcup D_n$. Perturb $\psi$ to 
a complex analytic map $\wt{\psi}$
such that $\wt{\psi}$ has only generic singularities, that is, positive double points. For small perturbation the links
$\LL_2'=\wt{\psi}(D)\cap \partial B$ and $\LL_1'=\wt{\psi}(D)\cap\partial B'$ are isotopic to $\LL_2$
and $\LL_1$ respectively. The APSIC is the intersection $\wt{\psi}(D)\cap (B\setminus B')$. The number of double monochromatic double
points is calculated as in \cite[Lemma 2.3]{BL2}; the argument is standard, we omit it.
\end{proof}

\end{document}